\documentclass[10pt,english]{smfart}

\usepackage[T1]{fontenc}
\usepackage[english,francais]{babel}
\usepackage[all,cmtip]{xy}

\usepackage{amssymb,url,xspace,smfthm}
\usepackage{amsmath}
\usepackage{color}
\usepackage{pstricks}
\usepackage{float}

\makeatother

\newcommand{\BibTeX}{{\scshape Bib}\kern-.08em\TeX}
\newcommand{\T}{\S\kern .15em\relax }
\newcommand{\AMS}{$\mathcal{A}$\kern-.1667em\lower.5ex\hbox
        {$\mathcal{M}$}\kern-.125em$\mathcal{S}$}

\author{Pedro \textsc{Montero}}
\address{Univ. Grenoble Alpes, Institut Fourier, F-38000 Grenoble, France.}
\email{pedro.montero@univ-grenoble-alpes.fr}
\title[Fano varieties with a divisor of Picard number one]{On singular Fano varieties with a divisor of Picard number one}

\newcommand{\Exc}{\operatorname{Exc}}
\newcommand{\Pic}{\operatorname{Pic}}
\newcommand{\NE}{\operatorname{NE}}
\newcommand{\N}{\operatorname{N}}
\newcommand{\Locus}{\operatorname{Locus}}
\newcommand{\Sing}{\operatorname{Sing}}
\newcommand{\codim}{\operatorname{codim}}
\newcommand{\redu}{\operatorname{red}}
\newcommand{\cont}{\operatorname{cont}}
\newcommand{\mult}{\operatorname{mult}}
\newcommand{\lisse}{\operatorname{reg}}
\newcommand{\h}{\operatorname{H}}

\newcounter{thm}
\newcounter{cor}
\newcounter{pro}
\newtheorem{thmy}{Theorem}
\newtheorem{coroy}{Corollary}
\newtheorem{propy}{Proposition}

\newenvironment{thmx}{\stepcounter{thm}\begin{thmy}}{\end{thmy}}
\newenvironment{corox}{\stepcounter{cor}\begin{coroy}}{\end{coroy}}
\newenvironment{propx}{\stepcounter{pro}\begin{propy}}{\end{propy}}
\begin{document}

\def\smfbyname{}

\begin{abstract}
In this paper we study the geometry of mildly singular Fano varieties on which there is an effective prime divisor of Picard number one. Afterwards, we address the case of toric varieties. Finally, we treat the lifting of extremal contractions to universal covering spaces in codimension 1.
\end{abstract}

\maketitle

\tableofcontents

\section{Introduction}

The aim of this article is to study the geometry of mildly singular Fano varieties on which there is a prime divisor of Picard number one. Recall that a Fano variety is a normal complex algebraic variety whose anti-canonical divisor has some positive multiple which is an ample Cartier divisor.

A first related result is given by L. Bonavero, F. Campana and J. A. Wi\'sniewski in the sequel of articles \cite{Bo02} and \cite{BCW02}, where the authors classified (toric) Fano varieties of dimension $n\geq 3$ on which there is a divisor isomorphic to $\mathbb{P}^{n-1}$ and later used these results to study (toric) complex varieties whose blow-up at a point is Fano. For instance, in the toric case we have the following result.

\begin{theo}[{\cite[Theorem 2]{Bo02}}]\label{thm: bonavero}
Let $X$ be a smooth toric Fano variety of dimension $n\geq 3$. Then, there exists a toric divisor $D$ of $X$ isomorphic to $\mathbb{P}^{n-1}$ if and only if one of the following situations occurs:
\begin{enumerate}
\item $X\cong \mathbb{P}^{n}$ and $D$ is a linear codimension 1 subspace of $X$.
\item $X\cong \mathbb{P}(\mathcal{O}_{\mathbb{P}^1}\oplus \mathcal{O}_{\mathbb{P}^1}(1)^{\oplus n-1	} )\cong \operatorname{Bl}_{\mathbb{P}^{n-2}}(\mathbb{P}^n)$ and $D$ is a fiber of the projection on $\mathbb{P}^1$.
\item $X\cong \mathbb{P}(\mathcal{O}_{\mathbb{P}^{n-1}}\oplus \mathcal{O}_{\mathbb{P}^{n-1}}(a))$, where $0\leq a \leq n-1$, and $D$ is either the divisor $\mathbb{P}(\mathcal{O}_{\mathbb{P}^{n-1}})$ or the divisor $\mathbb{P}(\mathcal{O}_{\mathbb{P}^{n-1}}(a))$.
\item $X$ is isomorphic to the blow-up of $\mathbb{P}(\mathcal{O}_{\mathbb{P}^{n-1}}\oplus \mathcal{O}_{\mathbb{P}^{n-1}}(a+1))$ along a linear $\mathbb{P}^{n-2}$ contained in the divisor $\mathbb{P}(\mathcal{O}_{\mathbb{P}^{n-1}})$, where $0\leq a \leq n-2$, and $D$ is either the strict transform of the divisor $\mathbb{P}(\mathcal{O}_{\mathbb{P}^{n-1}})$ or the strict transform of the divisor $\mathbb{P}(\mathcal{O}_{\mathbb{P}^{n-1}}(a+1))$.
\end{enumerate}
\end{theo}

In particular, this classification leads $\rho_X\leq 3$. Some years later, T. Tsukioka in \cite{Tsu06} used some arguments from \cite{And85} and \cite{BCW02} to generalize these results and proved, more generally, that a smooth Fano variety $X$ of dimension $n\geq 3$ containing an effective prime divisor of Picard number one must satisfy $\rho_X\leq 3$.

The bound $\rho_X\leq 3$ was recently proved by G. Della Noce in \cite[Remark 5.5]{DN14}, when $X$ is supposed to be a $\mathbb{Q}$-factorial Gorenstein Fano variety of dimension $n\geq 3$ with canonical singularities, with at most finitely many non-terminal points, and under the more general assumption of the existence of an effective prime divisor ${D\subseteq X}$ such that the real vector space $\N_1(D,X):=\operatorname{Im}\left(\N_1(D)\to \N_1(X) \right)$ of numerical classes of 1-cycles on $X$ that are equivalent to 1-cycles on $D$, is one-dimensional. Therefore, we will often consider varieties satisfying the following condition:

\vspace{2mm}

\begin{itemize}
 \item[($\dag$)] $X$ is a $\mathbb{Q}$-factorial normal variety of dimension $n\geq 3$ with Gorenstein canonical singularities, with at most finitely many non-terminal points.
\end{itemize}

\vspace{2mm}

In the smooth case, C. Casagrande and S. Druel provide in \cite{CD12} a classification (and examples) of all cases with maximal Picard number $\rho_X=3$.

\begin{theo}[{\cite[Theorem 3.8]{CD12}}] Let $X$ be a Fano manifold of dimension $n\geq 3$ and $\rho_X=3$. Let $D\subseteq X$ be a prime divisor with $\dim_\mathbb{R}\N_1(D,X)=1$. Then $X$ is isomorphic to the blow-up of a Fano manifold $Y\cong \mathbb{P}_Z(\mathcal{O}_Z\oplus \mathcal{O}_Z(a))$ along an irreducible subvariety of dimension $(n-2)$ contained in a section of the $\mathbb{P}^1$-bundle $\pi:Y\to Z$, where $Z$ is a Fano manifold of dimension $(n-1)$ and $\rho_Z=1$. 

\end{theo} 

Firstly, we recall in \textsection \ref{section:extremalcontractions} that a mildly singular Fano variety $X$ always has an extremal ray $R\subseteq \overline{\NE}(X)$ whose intersection with a given effective divisor is positive. The rest of \textsection \ref{section:extremalcontractions} is devoted to the study of these extremal contractions in the case that the given divisor has Picard number one. This allows us to prove the following result in \textsection \ref{section:rho=3}.

\begin{thmx}\label{thm:rho=3}
Let $X$ be a Fano variety satisfying $(\dag)$. Assume that there exists an effective prime divisor $D\subseteq X$ such that ${\dim_\mathbb{R}\N_1(D,X)=1}$ and that $\rho_X=3$. Then, there is a commutative diagram
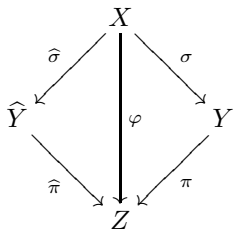
\[
\xymatrix{ & X \ar[ld]_{\widehat{\sigma}} \ar[rd]^{\sigma} \ar[dd]^\varphi & \\
\widehat{Y} \ar[rd]_{\widehat{\pi}} & & Y \ar[ld]^{\pi} \\
& Z &
}
\]
where $\sigma$ (resp. $\widehat{\sigma}$) corresponds to a divisorial contraction of an extremal ray ${R\subseteq \overline{\NE}(X)}$ (resp. $\widehat{R}\subseteq \overline{\NE}(X)$) which is the blow-up in codimension two of an irreducible subvariety of dimension $(n-2)$, and $\varphi$ is a contraction of fiber type, finite over $D$, corresponding to the face $R+\widehat{R}\subseteq \overline{\NE}(X)$. Moreover, $D\cdot R>0$, $Y$ and $\widehat{Y}$ are $\mathbb{Q}$-factorial varieties with canonical singularities and with at most finitely many non-terminal points, $Y$ is Fano and $Z$ is a $\mathbb{Q}$-factorial Fano variety with log terminal singularities. In particular, $Z$ has only rational singularities.

\end{thmx}

The results of S. Cutkosky on the contractions of terminal Gorenstein threefolds \cite{Cu88}, together with the previous result imply the following corollary.

\begin{corox}\label{thm:threefold}
Let $X$ be a $\mathbb{Q}$-factorial Gorenstein Fano threefold with terminal singularities. Assume that there exists an effective prime divisor $D\subseteq X$ such that $\dim_\mathbb{R}\N_1(D,X)=1$ and that $\rho_X=3$. Then, $X$ is factorial and it can be realized as the blow-up of a smooth Fano threefold $Y$ along a locally complete intersection curve $C\subseteq Y$. Moreover, $Y$ is isomorphic to $\mathbb{P}(\mathcal{O}_{\mathbb{P}^2}\oplus \mathcal{O}_{\mathbb{P}^2}(a))$, where $0\leq a \leq 2$.
\end{corox}

In the case $\rho_X=2$, we obtain in \textsection \ref{section:rho=2} an extension of \cite[Remark 3.2, Proposition 3.3]{CD12} to mildly singular Fano varieties $X$ with $\rho_X=2$, on which there is an effective prime divisor of Picard number one.

\setcounter{thmy}{2}

In order to extend the classification results to higher dimensions, we will restrict ourselves to the case of toric varieties. In that case, the combinatorial description of the MMP for toric varieties treated in \textsection \ref{section:toricmmp}, as well as some particular properties of them, will allow us to prove the following result in \textsection \ref{section:rho=3toric}.

\begin{thmx}\label{thm:rho=3toric}
Let $X$ be a toric Fano variety satisfying $(\dag)$. Assume that there exists an effective prime divisor $D\subseteq X$ such that ${\dim_\mathbb{R}\N_1(D,X)=1}$ and that ${\rho_X=3}$. Then, there exist $\mathbb{Q}$-factorial Gorenstein toric Fano varieties $Y$ and $Z$, with terminal singularities, such that
\begin{enumerate}
 \item $X\cong \overline{\operatorname{Bl}_A(Y)}$, the normalized blow-up of an invariant toric subvariety $A\subseteq Y$ of dimension $(n-2)$; and
 \item $Y\cong \mathbb{P}_Z(\mathcal{O}_Z\oplus \mathcal{O}_Z(a))$ with $0\leq a\leq i_Z-1$, where $i_Z$ is the Fano index of $Z$ and $\mathcal{O}_Z(1)$ is the ample generator of $\Pic(Z)$.
\end{enumerate}
If $\dim X \leq 4$, then $X$ is smooth and we are in the situation of Theorem \ref{thm: bonavero}, case (4).
\end{thmx}

In the toric setting, we obtain in \textsection \ref{section:rho=2toric} results that extend L. Bonavero's description of the extremal contractions in the case $\rho_X=2$ to mildly singular toric Fano varieties. If $X$ is supposed to have isolated canonical singularities then we obtain the following classification.

\begin{thmx}\label{thm:rho=2 toric div isolated}
 Let $X$ be a $\mathbb{Q}$-factorial Gorenstein toric Fano variety of dimension $n\geq 3$ with isolated canonical singularities. Assume that there exists an effective prime divisor $D\subseteq X$ such that $\dim_\mathbb{R}\N_1(D,X)=1$ and that $\rho_X=2$. Then, either
\begin{enumerate}
 \item $X\cong \mathbb{P}(\mathcal{O}_{\mathbb{P}^{n-1}}\oplus \mathcal{O}_{\mathbb{P}^{n-1}}(a))$, with $0\leq a \leq n-1$. In other words, we are in the situation of Theorem \ref{thm: bonavero}, case (3).
\item $X$ is isomorphic to the blow-up of a toric variety $Y$ along an invariant subvariety $A\subseteq Y$ of dimension $(n-2)$, contained in the smooth locus of $Y$. Moreover, $Y$ is isomorphic to either
\begin{itemize}
\item[(a)] $\mathbb{P}^n$,
\item[(b)] $\mathbb{P}(1^{n-1},2,n+1)$ if $n$ is even, or
\item[(c)] $\mathbb{P}(1^{n-1},a,b)$, where $1\leq a < b \leq n $ are two relatively prime integers such that $a|(n-1+b)$ and $b|(n-1+a)$.
\end{itemize}
In particular, $Y$ is a $\mathbb{Q}$-factorial Gorenstein Fano variety with $\rho_Y=1$ and it has at most two singular points. Conversely, the blow-up of any of the listed varieties $Y$ along an invariant irreducible subvariety $A\subseteq Y$ of dimension $(n-2)$ and contained in the smooth locus of $Y$, leads to a toric variety $X$ satisfying the hypothesis.
\end{enumerate}
\end{thmx}

Moreover, in the case of contractions of fiber type we obtain the following result without the assumption of isolated singularities.

\begin{propx}\label{thm:rho=2 toric fibration}
 Let $X$ be a toric Fano variety satisfying $(\dag)$. Assume that there exists an effective prime divisor $D\subseteq X$ such that $\dim_\mathbb{R}\N_1(D,X)=1$ and that $\rho_X=2$. Let $R\subseteq \overline{\NE}(X)$ be an extremal ray such that $D\cdot R > 0$ and assume that the corresponding extremal contraction $\pi:X\to Y$ is of fiber type. Then, $X \cong \mathbb{P}_Y(\mathcal{O}_Y\oplus \mathcal{O}_Y(a))$. Moreover, $Y$ is a $\mathbb{Q}$-factorial Gorenstein Fano variety of dimension $(n-1)$ with terminal singularities and Fano index $i_Y$, and $0\leq a \leq i_Y - 1$. In particular, $X$ has only terminal singularities.
\end{propx}

Finally, \textsection \ref{section:toric1coverings} is devoted to show that the extremal contractions studied in \textsection \ref{section:rho=2toric} lift to quasi-\'etale universal covers, introduced by W. Buczy\'nska \cite{Buc02} in order to study toric varieties of Picard number one. See Definition \ref{PWS} for the notion of Poly Weighted Space (PWS), introduced by M. Rossi and L. Terracini in \cite{RT15a} and proved to be quasi-\'etale universal covers for $\mathbb{Q}$-factorial toric varieties of arbitrary Picard number. 

In particular, we obtain the following description of divisorial contractions of toric mildly Fano varieties with Picard number two. It should be noticed that even if the combinatorial description of these divisorial contractions is very simple (see Lemma \ref{toricdivisorial}) and it coincides with the one of the blow-up of a subvariety of dimension $(n-2)$ in the smooth case, it may happen that the morphisms are not globally a blow-up of the coherent sheaf of ideals of a (irreducible and reduced) subvariety but only a blow-up in codimension two if the singularities are not isolated (see Example \ref{contreexample}).

\begin{propx}\label{prop:1-covering,divisorial}
 Let $X$ be a toric Fano variety satisfying $(\dag)$. Assume that there exists an effective prime divisor $D\subseteq X$ such that $\dim_\mathbb{R}\N_1(D,X)=1$ and that $\rho_X=2$. Let $R\subseteq \overline{\NE}(X)$ be an extremal ray such that $D\cdot R > 0$ and let us denote by $\pi:X\to Y$ the corresponding extremal contraction. Assume that $\pi$ is birational. Then there exist weights $\lambda_0,\ldots,\lambda_n\in \mathbb{Z}_{>0}$ and a cartesian diagram of toric varieties
$$
\xymatrix{
\widehat{X} \ar[r]^{\widehat{\pi}\hspace{1cm}} \ar[d]_{\pi_X} & \mathbb{P}(\lambda_0,\ldots,\lambda_{n}) \ar[d]^{\pi_Y} \\
X \ar[r]^\pi & Y
}
$$
where vertical arrows denote the corresponding canonical quasi-\'etale universal covers, and $\widehat{X}$ is a Fano PWS satisfying $(\dag)$ such that $\rho_{\widehat{X}}=2$. Moreover, ${\widehat{\pi}:\widehat{X}\to \mathbb{P}(\lambda_0,\ldots,\lambda_{n})}$ is a divisorial contraction sending its exceptional divisor ${\widehat{E}\subseteq \widehat{X}}$ onto an invariant subvariety $\widehat{A}\subseteq \mathbb{P}(\lambda_0,\ldots,\lambda_{n})$ of dimension $(n-2)$.
\end{propx}

\subsubsection*{{\bf Acknowledgements}}

I would like to thank my supervisors, St\'ephane \textsc{Druel} and Catriona \textsc{Maclean}, for their unstinting help and their valuable guidance. I also thank Yuri \textsc{Prokhorov} and Mikhail \textsc{Zaidenberg} for fruitful discussions, as well as Miles \textsc{Reid} for kindly communicating me that the restriction of a toric extremal contraction to its exceptional locus may not be flat. Last, but not least, I would like to express my gratitude to the anonymous referee for the suggestions and comments on the first version of this article.

\section{Notation and preliminary results}

Through out this article all varieties will be assumed to be reduced and irreducible schemes of finite type over the field $\mathbb{C}$ of complex numbers\footnote{Most of the results are true over an arbitrary algebraically closed field of characteristic zero, but for Corollary \ref{D not ample} to hold, the field must be also uncountable.}. Its smooth locus will be denoted by $X_{\lisse}\subseteq X$, while $\Sing(X)=X\setminus X_{\lisse}$ denotes its singular locus. 

We will use the notation and results of the Minimal Model Program (MMP for short) in \cite{KM98}. For instance, if $X$ is a normal projective variety then we denote by $\N_1(X)$ the real vector space of 1-cycles with real coefficients modulo numerical equivalence and $\overline{\NE}(X)\subseteq \N_1(X)$ the {\it Mori cone}, which is the closed and convex cone generated by numerical classes of effective 1-cycles. We denote by $[C]$ the class of $C$ in $\N_1(X)$.

Let $Z\subseteq X$ be a closed subset and $\iota:Z \to X$ be the natural inclusion, we define
$$\N_1(Z,X)=\iota_*\N_1(Z)\subseteq \N_1(X). $$ 

For us, a divisor will always be a Weil divisor. Let us denote by $K_X$ the class of a canonical divisor in $\operatorname{Cl}(X)$. A complete normal variety $X$ is said to be a Fano variety if there exists a positive multiple of $-K_X$ which is Cartier and ample. 

We follow the usual convention, and we say that $X$ is a $\mathbb{Q}$-Gorenstein variety if some positive multiple of $K_X$ is a Cartier divisor; we do not require Cohen-Macaulay singularities unless $K_X$ is already Cartier. In this case, the Gorenstein index of $X$ is the smallest positive integer $\ell\in \mathbb{Z}_{>0}$ such that $\ell K_X$ is a Cartier divisor. 

In the same way as for $1-$cycles, we can define $\N^1(X)$ as the vector space of $\mathbb{Q}$-Cartier divisors with real coefficients, modulo numerical equivalence $\equiv$. We denote by $[D]$ the class of $D$ in $\N^1(X)$. We have that $\dim_\mathbb{R}\N_1(X)=\dim_\mathbb{R}\N^1(X)=\rho_X$, which is called the Picard number of $X$.

We denote
$$\Locus(R)=\bigcup_{[C]\in R}C \subseteq X. $$

If $D$ is a $\mathbb{Q}$-Cartier divisor and $R\subseteq \overline{\NE}(X)$ is an extremal ray, then the sign $D\cdot R>0$ (resp. $=, <$) is well defined.

A contraction of $X$ is a projective surjective morphism $\varphi:X\to Y$ with connected fibers, where $Y$ is a normal projective variety. In particular, $\varphi_* \mathcal{O}_X = \mathcal{O}_Y$.

We refer the reader to \cite[\textsection 2.3]{KM98} for the notion of singularities of pairs for $\mathbb{Q}$-factorial normal varieties and to \cite[Chapter 3]{KM98} for the Cone Theorem. For instance, we say that a $\mathbb{Q}$-factorial normal variety $X$ is log terminal if $(X,0)$ is a klt pair.

In the case when $X$ is a log terminal Fano variety the Mori cone of $X$ is finite rational polyhedral, generated by classes of rational curves; in particular $\overline{\NE}(X)=\NE(X)$. Moreover, Kawamata-Viehweg theorem implies in that case that $\h^i(X,\mathcal{O}_X)=\{0\}$ for $i>0$. After the works of Birkar, Cascini, Hacon and McKernan, if $X$ is a log terminal Fano variety, then $X$ is a Mori Dream Space (see \cite[Corollary 1.3.2]{BCHM10} and \cite{HK00}).

For us, a $\mathbb{P}^1$-bundle is a smooth morphism all of whose fibers are isomorphic to $\mathbb{P}^1$. 

Let $E$ be a vector bundle over a variety $X$, we denote by $\mathbb{P}_X(E)$ the scheme $\mathbf{Proj}_X \left(\bigoplus_{d\geq 0}S^d(E) \right)$.

Finally, we recall that if $\lambda_0,\ldots,\lambda_n$ are positive integers with $\gcd(\lambda_0,\ldots,\lambda_n)=1$ then we define the associated Weighted Projective Space (WPS) $\mathbb{P}(\lambda_0,\ldots,\lambda_n)$, which is a toric variety with Picard number one and torsion-free class group. In general, we say that a $\mathbb{Q}$-factorial complete toric variety is a Poly Weighted Space (PWS) if its class group is torsion-free.

\section{Study of the extremal contractions}\label{section:extremalcontractions}

In this section we study extremal contractions of mildly singular Fano varieties $X$ that admits an effective prime divisor $D\subseteq X$ such that $\dim_\mathbb{R}\N_1(D,X)=1$. 

Firstly, notice that log terminal Fano varieties always have an extremal ray whose intersection with a given prime divisor is positive (c.f. \cite[Lemma 2]{BCW02}).

\begin{lemm}\label{Rexists}

Let $X$ be a $\mathbb{Q}$-factorial log terminal Fano variety and let $D\subseteq X$ be an effective prime divisor. Then, there exists an extremal ray $R\subseteq \overline{\NE}(X)$ such that $D\cdot R > 0$.

\end{lemm}

\begin{proof}
The Cone Theorem \cite[Theorem 3.7]{KM98} implies that $\overline{\NE}(X)=\NE(X)$ is a rational polyhedral cone generated by a finite number of extremal rays $R_1,\ldots,R_s$. Let $C\subseteq X$ be any curve such that $D\cdot [C]>0$. Since $C$ is numerically equivalent to a positive sum of extremal curves, $[C]=\sum_{i=1}^s a_i [C_i]$ with $a_i\geq 0$ and $[C_i]\in R_i$, we can pick one such that $D\cdot R_i>0$.
\end{proof}

Secondly, we have that the contraction of an extremal ray whose intersection with an effective prime divisor of Picard number one is positive has at most one-dimensional fibers.

\begin{lemm}\label{fibersdim1}

Let $X$ be a $\mathbb{Q}$-factorial log terminal Fano variety and $D\subseteq X$ be an effective prime divisor such that $\dim_\mathbb{R}\N_1(D,X)=1$. Let us suppose that $\rho_X > 1$ and let $R\subseteq \overline{\NE} (X)$ be an extremal ray such that $D\cdot R>0$. Then, $R \nsubseteq \N_1(D,X)$. In particular, the extremal contraction $\cont_R$ is finite on $D$ and all the fibers of $\cont_R$ are at most of dimension 1.

\end{lemm}

\begin{proof}
The proof in the smooth case \cite[Lemma 3.1]{CD12} generalizes verbatim to this setting.
\end{proof}

G. Della Noce proved in \cite[Theorem 2.2]{DN14} that a projective variety (not necessarily Fano) satisfying ($\dag$) has no small $K$-negative extremal contractions having all the fibers of dimension at most 1 (cf. \cite[Example 2.11]{DN14}). In our context we have the following result. 

\begin{prop}\label{extcont}
Let $X$ be a $\mathbb{Q}$-factorial log terminal Fano variety of dimension $n\geq 3$. Let us suppose that $\rho_X>1$ and that there exists an effective prime divisor such that $\dim_\mathbb{R}\N_1(D,X)=1$. Let $R\subseteq \overline{\NE}(X)$ be an extremal ray such that $D\cdot R>0$ and let us denote by $\varphi_R:X\to X_R$ the corresponding extremal contraction. Then,
\begin{enumerate}
\item If $\varphi_R$ is of fiber type, then $X_R$ is a $\mathbb{Q}$-factorial log terminal Fano variety of dimension $(n-1)$ such that $\rho_{X_R}=1$. In particular, $X_R$ has only rational singularities.
\item If $\varphi_R$ is birational and we suppose that $X$ satisfies $(\dag)$, then $\varphi_R$ is a divisorial contraction and there exists a closed subset $S\subseteq X_R$ with $\codim_{X_R}(S)\geq 3$ such that $X_R \smallsetminus S \subseteq X_{R,\lisse}$, $\codim_X \varphi_R^{-1}(S)\geq 2$, $X\smallsetminus \varphi^{-1}(S) \subseteq X_{\lisse}$ and
  $$\varphi_R|_{X\smallsetminus \varphi_R^{-1}(S)}:X\smallsetminus \varphi_R^{-1}(S) \longrightarrow X_R \smallsetminus S$$
is the blow-up of a $(n-2)$-dimensional smooth subvariety in $X_R\smallsetminus S$. Moreover, $X_R$ is a $\mathbb{Q}$-factorial Fano variety with canonical singularities with at most finitely many non-terminal points. In particular, 
$$K_X\cdot [F]=E\cdot [F]=-1,$$
for every irreducible curve $F$ such that $[F]\in R$, where $E=\Exc(\varphi_R)$ is the exceptional divisor of $\varphi_R$.
\end{enumerate}
\end{prop}

\begin{proof}
Let us suppose that $\varphi_R:X\to X_R$ is of fiber type. Then, $\varphi_R|_D:D\to X_R$ is a finite morphism, by Lemma \ref{fibersdim1}, and thus $\dim_\mathbb{C} X_R=n-1$. Since $X_R$ is a $\mathbb{Q}$-factorial projective variety, the projection formula and the fact that $\dim_\mathbb{R}\N_1(D,X)=1$ imply that $\rho_{X_R}=1$. In particular, every big divisor on $X_R$ is ample.

Let us prove that $X_R$ is a log terminal Fano variety in this case. Since $X$ is log terminal and Fano, there exists an effective $\mathbb{Q}$-divisor $\Delta_{X_R}$ on $X_R$ such that $(X_R,\Delta_{X_R})$ is klt and such that $-\left( K_{X_R}+\Delta_{X_R}\right)$ is ample, by \cite[Lemma 2.8]{PS09}. In particular $X_R$ has rational singularities, by \cite[Theorem 5.22]{KM98}, and $-K_{X_R}$ is a big divisor, and therefore ample. Finally, let us remark that since $(X_R,\Delta_{X_R})$ is klt and $\Delta_{X_R}$ is an effective $\mathbb{Q}$-divisor then it follows from \cite[Corollary 2.35]{KM98} that $(X_R,0)$ is klt as well. 

Let us suppose now that $X$ satisfies $(\dag)$ and that $\varphi_R:X\to X_R$ is a birational contraction. Then, it follows from \cite[Theorem 2.2]{DN14} that $\varphi_R$ is a divisorial contraction given by the blow-up in codimension two of an irreducible subvariety of dimension $(n-2)$ on $X_R$, and that $X_R$ is a $\mathbb{Q}$-factorial Fano variety with canonical singularities with at most finitely many non-terminal points. Finally, the ampleness of the anti-canonical divisor $-K_{X_R}$ follows verbatim from the proof in the smooth case given in \cite[Lemma 3.1]{CD12}.
\end{proof}

\begin{rema}\label{rholeq3}
Let $X$ be a Fano variety satisfying $(\dag)$ and let $D\subseteq X$ be an effective prime divisor such that $\dim_\mathbb{R}\N_1(D,X)=1$. Then $\rho_X\leq 3$, by \cite[Remark 5.5]{DN14}.
\end{rema}

\begin{rema}[Generalized conic bundles]\label{fibersP1}
 Following G. Della Noce \cite[p. 984]{DN14}, we say that a morphism $\varphi:X\to Y$ between $\mathbb{Q}$-factorial normal projective varieties is a {\it generalized conic bundle} if all its fibers are one-dimensional, the general fiber $F_g$ is isomorphic to $\mathbb{P}^1$ with anti-canonical degree $-K_X\cdot [F_g] = 2$ and if $F$ is an arbitrary fiber of $\varphi$ then either:
 \begin{enumerate}
\item $F$ is an irreducible and generically reduced rational curve such that $F_{\redu} \cong \mathbb{P}^1$ and that $-K_X\cdot [F]=2$.

\item $[F]=2[C]$ as 1-cycles, where $C$ is an irreducible and generically reduced rational curve such that $C_{\redu} \cong \mathbb{P}^1$ and that $-K_X\cdot [C] = 1$.

\item $F=C\cup C'$, with $C\neq C'$ irreducible and generically reduced rational curves such that $C_{\redu}\cong C'_{\redu}\cong \mathbb{P}^1$ and that $-K_X\cdot [C] = -K_X\cdot [C'] = 1$.
\end{enumerate}
 Moreover, it follows from \cite[Theorem II.2.8]{Kol96} that if all the fibers are of type (1) then $\varphi$ is in fact a $\mathbb{P}^1$-bundle. The main difference with the classical conic bundle case is that $\varphi$ might not be flat. As observed in \cite[p. 984]{DN14} if $X$ is a Fano variety satisfying $(\dag)$ and 
 $$X=X_0 \overset{\sigma_0}{\rightarrow} X_1 \rightarrow \cdots \rightarrow X_{k-1} \overset{\sigma_{k-1}}{\rightarrow} X_k \overset{\pi}{\rightarrow} Y$$
 is a sequence of $k\geq 0$ elementary divisorial contractions $\sigma_i:X_i\to X_{i+1}$ sending its corresponding exceptional divisor onto a subvariety of codimension two (cf. \cite[Definition 2.10]{DN14}) followed by a contraction of fiber type $\pi:X_k\to Y$ all whose fibers are one-dimensional, then $\pi$ is a generalized conic bundle in the above sense and, moreover, for every $i=0,\ldots,k-1$, the composition 
 $$\varphi_i:X_i \overset{\sigma_i}{\rightarrow} X_{i+1} \rightarrow \cdots \rightarrow X_{k-1} \overset{\sigma_{k-1}}{\rightarrow} X_k \overset{\pi}{\rightarrow} Y$$
 is a generalized conic bundle. We will be mainly interested in the cases $k=0,1$.
\end{rema}

Let us finish the section with two results concerning birational extremal contractions.

\begin{theo}[{\cite[Theorem 3.1]{DN14}}]\label{specialMoriprogram} Let $X$ be a $\mathbb{Q}$-factorial Fano variety with canonical singularities. Then, for any prime divisor $D\subseteq X$, there exists a finite sequence (called a special Mori program for the divisor $-D$)
$$X=X_0 \overset{\sigma_0}{\dashrightarrow} X_1 \dashrightarrow \cdots \dashrightarrow X_{k-1} \overset{\sigma_{k-1}}{\dashrightarrow} X_k \overset{\pi}{\rightarrow} Y  $$
such that, if $D_i\subseteq X_i$ is the transform of $D$ for $i=1,\ldots,k$ and $D_0:=D$, the following hold:
\begin{enumerate}
\item $X_1,\ldots, X_k$ and $Y$ are $\mathbb{Q}$-factorial projective varieties and $X_1,\ldots,X_k$ have canonical singularities.
\item for every $i=0,\ldots, k$, there exists an extremal ray $Q_i$ of $X_i$ with $D_i\cdot Q_i>0$ and $-K_{X_i}\cdot Q_i > 0$ such that:
\begin{itemize}
\item[(a)] for $i=1,\ldots,k-1$, $\Locus(Q_i)\subsetneq X_i$, and $\sigma_i$ is either the contraction of $Q_i$ (if $Q_i$ is divisorial), or its flip (if $Q_i$ is small);
\item[(b)] the morphism $\pi :X_k \to Y$ is the contraction of $Q_k$ and $\pi$ is a fiber type contraction.
\end{itemize}
\end{enumerate}
\end{theo}

The following result is a particular case of \cite[Lemma 3.3]{DN14}. We include the statement with our notation for completeness.

\begin{lemm} \label{Gorlocus}
Let $X_0$ be a Fano variety satisfying $(\dag)$. Let $D_0\subseteq X_0$ be an effective prime divisor and let us suppose that there is a diagram
$$X_0 \overset{\sigma_0}{\dashrightarrow} X_1 \overset{\sigma_1}{\dashrightarrow} X_2, $$
where $\sigma_i$ is the birational map associated to the contraction of an extremal ray $Q_i\subseteq \overline{\NE}(X_i)$ such that $D_i\cdot Q_i>0$, for $i=0,1$; as in Theorem \ref{specialMoriprogram}. If $Q_i \not \subseteq \N_1(D_i,X_i)$ for $i=0,1$, then both $\sigma_0$ and $\sigma_1$ are divisorial contractions, $\Exc(\sigma_i)$ is contained in the Gorenstein locus of $X_i$ and $\Exc(\sigma_0)$ is disjoint from the transform of $\Exc(\sigma_1)$ in $X_0$.  
\end{lemm}

\section{The extremal case $\rho_X=3$}\label{section:rho=3}
In this section we study the extremal contractions of mildly singular Fano varieties $X$ on which there is an effective prime divisor of Picard number one and such that $\rho_X=3$. As we pointed out in Remark \ref{rholeq3}, this is the largest possible Picard number for such varieties. Compare with the smooth case \cite[Lemma 3.1, Theorem 3.8]{CD12}.

\begin{proof}[Proof of Theorem \ref{thm:rho=3}]

Since $X$ is a $\mathbb{Q}$-factorial log terminal Fano variety, there is an extremal ray $R\subseteq \overline{\NE}(X)$ such that $D\cdot R>0$, by Lemma \ref{Rexists}. We denote by $\sigma:=\varphi_R:X\to Y$ the associated extremal contraction. We note that Proposition \ref{extcont} implies that $\sigma$ is a divisorial contraction sending the effective prime divisor $E=\Exc(\sigma)$ onto a subvariety $A\subseteq Y$ of dimension $(n-2)$. Moreover, $Y$ is a Fano variety with canonical singularities and with at most finitely many non-terminal points, such that if we denote by $D_Y$ the image of $D$ by $\sigma$, then we have that $A\subseteq D_Y$ and that $\dim_\mathbb{R} \N_1(D_Y,Y)=1$.

Since $Y$ is $\mathbb{Q}$-factorial log terminal Fano variety, there is an extremal ray $Q\subseteq \overline{\NE}(Y)$ such that $D_Y\cdot Q>0$, by Lemma \ref{Rexists}. We denote by $\pi:=\varphi_Q:Y\to Z$ the associated extremal contraction. Let us prove that $\pi$ is of fiber type. Assume, to the contrary, that $\pi$ is a birational contraction. Hence, Lemma \ref{Gorlocus} implies that both $\sigma$ and $\pi$ are divisorial contractions, and that the exceptional locus $\Exc(\pi\circ \sigma)$ consists of two disjoint effective prime divisors. Since $D_Y\cdot Q > 0$ we have that $\Exc(\pi)\cdot [C]>0$ for every irreducible curve $C\subseteq D_Y$. In particular, $\Exc(\pi)\cap A \not= \emptyset$, as $\dim_\mathbb{C}A = n-2 \geq 1$, contradicting the fact that the exceptional divisors are disjoint.

Let $\widehat{R}\subseteq \overline{\NE}(X)$ be the extremal ray such that $\cont_{R+\widehat{R}}=\varphi:=\pi\circ \sigma$. Then, $\varphi$ can be factorized as $\varphi=\widehat{\pi}\circ\widehat{\sigma}$, where $\widehat{\sigma}:=\cont_{\widehat{R}}:X\to \widehat{Y}$ and $\widehat{\pi}:=\cont_{\widehat{\pi}(R)}:\widehat{Y}\to Z$. Since $\varphi$ has fibers of dimension 1, both contractions must have fibers of dimension at most 1. Notice that the general fiber of $\varphi$ is not contracted by $\widehat{\sigma}$, hence $\widehat{\sigma}$ must be divisorial and $\widehat{\pi}$ a contraction of fiber type, by the same arguments as above. 
\end{proof}

The results of S. Cutkosky on the contractions of terminal Gorenstein threefolds and Theorem \ref{thm:rho=3} above lead to Corollary \ref{thm:threefold}.

\begin{proof}[Proof of Corollary \ref{thm:threefold}]

By Theorem \ref{thm:rho=3}, there is a diagram
$$\xymatrix{X \ar[r]^{\sigma} \ar@/_1pc/[rr]_{\varphi} & Y \ar[r]^{\pi} & Z}, $$
where $\sigma:X\to Y$ is a divisorial contraction sending a prime divisor $E=\Exc(\sigma)\subseteq X$ onto a curve $C \subseteq Y$, and $\pi$ and $\varphi$ are both extremal contractions of fiber type whose fibers are of dimension 1. All these varieties are $\mathbb{Q}$-factorial Fano varieties, and $Y$ has terminal singularities. Moreover, $X$ is factorial by \cite[Lemma 2]{Cu88}.

By \cite[Lemma 3, Theorem 4]{Cu88}, $C\subseteq Y$ is an irreducible reduced curve which is a locally complete intersection, $Y$ is a factorial threefold which is smooth near the curve $C$, and $\sigma:X\to Y$ is the blow-up of the ideal sheaf $\mathcal{I}_C$. In particular, $Y$ is a Gorenstein Fano threefold with terminal singularities and therefore \cite[Theorem 7]{Cu88} implies that $Z$ is a smooth surface and $\pi:Y\to Z$ is a (possibly singular) conic bundle over $Z$. We note that $Z$ is a rationally connected (and therefore rational) surface with $\rho_Z=1$, hence $Z\cong \mathbb{P}^2$.

Let $H=\sigma(C) \subseteq Z$. Since $\pi$ is finite on $D_Y=\sigma(D)$ and $C\subseteq D_Y$, we have that $H$ is an effective prime divisor on $\mathbb{P}^2$, which is therefore ample. Let us denote by $S_\pi$ the locus of points of $\mathbb{P}^2$ over which $\pi$ is not a smooth morphism. By \cite[Proposition II.1.1]{Gr71}, is a closed subset of $\mathbb{P}^2$. Then, $S_\pi$ has pure codimension 1 on $\mathbb{P}^2$ or $S_\pi=\emptyset$, by \cite[Theorem 3]{ArRM}.

Let us suppose that $S_\pi$ is not empty. If we take $z\in H\cap S_\pi$ and we denote by $F_z\subseteq Y$ its fiber by $\pi$, then we have that $F_z\cap C\neq \emptyset$ (as $z\in H$) and that the 1-cycle on $Y$ associated to $F_z$ is of the form $[F_z]=[C]+[C']$, where $C$ and $C'$ are (possibly equal) irreducible and generically reduced rational curves such that $C_{\redu}\cong C'_{\redu} \cong \mathbb{P}^1$ (as $z\in S_\pi$), by Remark \ref{fibersP1}. Thus, if we denote by $\tilde{F}_z\subseteq X$ the total transform of $F_z$ on $X$ by $\sigma$, we will have $-K_X\cdot [\tilde{F}_z]\geq 3$, contradicting the fact that the anti-canonical degree of every fiber of $\varphi=\pi\circ \sigma$ is 2 (see Remark \ref{fibersP1}). We conclude in this way that $\pi:Y\to \mathbb{P}^2$ is a $\mathbb{P}^1$-bundle, and then $Y$ is an smooth threefold by \cite[Theorem 5]{ArRM}. 

Finally, let us notice that if $D_Y$ is not nef then there is a birational contraction $Y\to Y_0$ sending $D_Y$ to a point, by \cite[Remark 3.2]{CD12}, and hence \cite[Lemma 3.9]{CD12} implies that $Y\cong \mathbb{P}(\mathcal{O}_{\mathbb{P}^2}\oplus \mathcal{O}_{\mathbb{P}^2}(a))$, with $0 \leq a \leq 2$ because $Y$ is Fano. On the other hand, if $D_Y$ is nef then we apply \cite[Proposition 3.3]{CD12} to conclude that either there is a divisorial contraction $Y\to Y_0$ sending an effective prime divisor $G_Y\subseteq Y$ to a point, or there is a contraction of fiber type $Y\to \mathbb{P}^1$; small contractions are excluded since $Y$ is a smooth Fano threefold (see for instance \cite{Cu88} or \cite[Theorem 1.32]{KM98}). In the first case \cite[Lemma 3.9]{CD12} allows us to conclude, while in the second case we have that $Y\cong \mathbb{P}^2\times \mathbb{P}^1$, by \cite[Lemma 4.9]{Cas09}.
\end{proof}

\section{The case $\rho_X=2$}\label{section:rho=2}

In the case $\rho_X=2$ we can describe the extremal contraction associated to the other extremal ray in the Mori cone of $X$ (compare with \cite[Remark 3.2, Proposition 3.3]{CD12}). We will need the following result of G. V. Ravindra and V. Srinivas (see \cite{RS06}).

\begin{theo} \label{lefs}
Let $X$ be a complex normal projective variety and let $\mathcal{L}$ be an ample and globally generated line bundle over $X$. Then there is a dense Zariski open set of divisors $E\in |\mathcal{L}|$ such that the restriction map 
$$\operatorname{Cl}(X) \rightarrow \operatorname{Cl}(E) $$
is an isomorphism, if $\dim_\mathbb{C} X \geq 4$, and is injective, with finitely generated cokernel, if $\dim_\mathbb{C} X=3$.
\end{theo}

\begin{coro} \label{coro1}\label{D not ample}
Let $X$ be a $\mathbb{Q}$-factorial Fano variety of dimension $n\geq 3$ with log terminal singularities. If $\rho_X>1$ and $D$ is an effective prime divisor such that $\dim_\mathbb{R}\N_1(D,X)=1$, then $D$ is not an ample divisor.
\end{coro}

\begin{proof}

Assume, to the contrary, that $D$ is an ample divisor. Let $m\in \mathbb{Z}_{>0}$ such that $mD$ is a very ample Cartier divisor on $X$ and use the complete linear system $|mD|$ to embed $X\hookrightarrow \mathbb{P}(\h^0(X,mD))=\mathbb{P}^N$. Let us define the projective incidence variety
$$\mathfrak{D}=\{(x,[E])\in X\times |mD|\;|\;x\in E\} \subseteq X \times |mD|= X \times (\mathbb{P}^N)^*, $$
and let $\pi:\mathfrak{D}\to (\mathbb{P}^N)^*$ be the second projection.

Let $\mathcal{H}$ be the relative Hilbert scheme of curves associated to the morphism $\pi:\mathfrak{D}\to (\mathbb{P}^N)^*$, which is a projective scheme with countably many irreducible components. Let us denote by $Z_i\subseteq (\mathbb{P}^N)^*$ the image of the components of $\mathcal{H}$ that do not dominate $(\mathbb{P}^N)^*$. They are closed subsets of $(\mathbb{P}^N)^*$.

Thus, if we take $[E]\in (\mathbb{P}^N)^*-\cup Z_i$ (a very general point on $(\mathbb{P}^N)^*$), then for every curve $C_E$ on $E$ there is a dominant component of $\mathcal{H}$ such that $C_E$ is one the curves parametrized by this component. Since the image of this dominant component is in fact the whole projective space $(\mathbb{P}^N)^*$, then we have that there is a curve $C_D$ on $D$ which is also parametrized for this component. In particular, $C_E$ and $C_D$ are numerically equivalent. But, since there is only one curve on $D$ up to numerical equivalence, we will have $\dim_\mathbb{R}\N_1(E,X)=1$.

On the other hand, Theorem \ref{lefs} implies that we can also suppose that this very general divisor $E\in |mD|$ is chosen in such a way the restriction 
$$\operatorname{Cl}(X) \rightarrow \operatorname{Cl}(E) $$
is injective. Since $\rho_X>1$ and $\dim_\mathbb{R}\N_1(E,X)=1$, we have that the inclusion $\N_1(E)\to \N_1(X)$ is not surjective. Therefore, the induced map on the dual spaces obtained by restriction $\N^1(X)\to \N^1(E)$ is not injective.

Since $X$ is a $\mathbb{Q}$-factorial log terminal Fano variety, we have that numerical and linear equivalence coincide, by \cite[Lemma 2.5]{AD14}. Hence, we have the following diagram
$$\xymatrix{
\operatorname{Cl}(X)\otimes_{\mathbb{Z}}\mathbb{R} \ar@{^{(}->}[d] & \operatorname{Pic}(X)\otimes_{\mathbb{Z}}\mathbb{R} \ar@{_{(}->}[l]_{\cong} \ar@{->>}[r]^{\hspace{3mm}\cong} \ar@{^{(}->}[d] & \N^1(X) \ar@{->>}[d] \\
\operatorname{Cl}(E)\otimes_{\mathbb{Z}}\mathbb{R}  & \operatorname{Pic}(E)\otimes_{\mathbb{Z}}\mathbb{R} \ar@{_{(}->}[l] \ar@{->>}[r] & \N^1(E)
}$$
Therefore, $\N^1(X)\to \N^1(E)$ is not injective if and only if $\operatorname{Pic}(E)\otimes_{\mathbb{Z}}\mathbb{R} \to \N^1(E)$ is not injective. 

Let us consider a line bundle $\mathcal{L}$ on $E$ such that $\mathcal{L}\equiv 0$. Notice that $E$ has log terminal singularities since it is a general member of the ample linear system $|mD|$, by \cite[Lemma 5.17]{KM98}. In particular, $E$ has rational singularities and thus if we consider a resolution of singularities $\varepsilon:\widetilde{E}\to E$, then the Leray spectral sequence leads to $h^i(\widetilde{E},\mathcal{O}_{\widetilde{E}})=h^i(E,\mathcal{O}_E)$ and $h^i(E,\mathcal{L})=h^i(\widetilde{E},\varepsilon^*\mathcal{L})$ for all $i\geq 0$. On the other hand, the short exact sequence of sheaves
$$ 0 \rightarrow \mathcal{O}_X(-E) \rightarrow \mathcal{O}_X \rightarrow \mathcal{O}_E \rightarrow 0 $$
and the vanishing $h^i(X,\mathcal{O}_X)=0$ for $i\geq 1$, give us $h^1(E,\mathcal{O}_E)=h^2(X,\mathcal{O}_X(-E))$. Since $X$ is a normal variety with Cohen-Macaulay singularities, Serre's duality implies
$$\h^2(X,\mathcal{O}_X(-E))\cong \h^2(X,\mathcal{O}_X(-mD))\cong \h^{n-2}(X,\mathcal{O}_X(K_X+mD))^\vee. $$
So, by taking $m$ large enough at the beginning if necessary, we can suppose that $h^1(E,\mathcal{O}_E)=0$ by the Kawamata-Viehweg vanishing theorem.

We get that $h^1(\widetilde{E},\mathcal{O}_{\widetilde{E}})=0$ and hence an inclusion $\operatorname{Pic}(\widetilde{E}) \hookrightarrow \h^2(\widetilde{E},\mathbb{Z})$. Clearly $\varepsilon^*\mathcal{L}\equiv 0$ and thus $\varepsilon^*\mathcal{L} \cong \mathcal{O}_{\widetilde{E}}$. By the projection formula, $\mathcal{L}\cong \mathcal{O}_E$, a contradiction.
\end{proof}

We end this section by proving the following result.

\begin{theo}\label{thm:rho=2,2nd contr}
Let $X$ be a $\mathbb{Q}$-factorial Gorenstein Fano variety of dimension ${n\geq 3}$ with canonical singularities and with at most finitely many non-terminal points. Assume that there exists an effective prime divisor $D\subseteq X$ such that ${\dim_\mathbb{R}\N_1(D,X)=1}$ and that $\rho_X=2$. There are two possibilities:
\begin{enumerate}
\item If $D$ is not nef, then there is an extremal contraction sending $D$ to a point.
\item If $D$ is nef, then $S=D^\perp\cap \NE(X)$ is an extremal ray. One of the following assertions must hold:
  \begin{itemize}
  \item[a)] $\cont_S$ is of fiber type onto $\mathbb{P}^1$, and $D$ is a fiber.
  \item[b)] $\cont_S$ is a divisorial contraction sending its exceptional divisor $G$ to a point, and such that $G\cap D=\emptyset$.
  \item[c)] $\cont_S$ is a small contraction and there is a flip $X \dashrightarrow X'$ and a contraction of fiber type $\psi: X'\to Y'$ such that the general fiber is isomorphic to $\mathbb{P}^1$, with anti-canonical degree 2. Moreover, $\psi$ is finite over the strict transform of $D$ in $X'$.
  \end{itemize}
\end{enumerate}
\end{theo} 

\begin{proof}
If $D$ is not nef, then there exists an extremal ray $R\subseteq \NE(X)$ such that $D\cdot R<0$, and therefore $\Exc(\cont_R)\subseteq D$. Since $\dim_\mathbb{R}\N_1(D,X)=1$ we must have that $\Exc(\cont_R)=D$ and that $\cont_R(D)$ is a point.

If $D$ is nef, then it is not ample by Corollary \ref{coro1}, and thus $S=D^\perp\cap \NE(X)$ is an extremal ray, as $\rho_X=2$. If we denote $G=\Locus(S)\subseteq X$, the proof follows almost verbatim the proof in the smooth case given in \cite[Proposition 3.3]{CD12}. For reader's convenience, let us briefly sketch the proof: 

If $S\subseteq \N_1(D,X)$ then $\cont_S$ sends $D$ to a point. On the other hand, $D\cdot S = 0$ and hence $D$ is the pullback of a divisor via $\cont_S$, from which we conclude that the target of $\cont_S$ is $\mathbb{P}^1$ and hence that we are in case (a). 

If $S\not\subseteq \N_1(D,X)$ we verify that $G\cap D = \emptyset$ and hence $\cont_S$ is birational. Moreover $\N_1(G,X)\subseteq D^\perp$ has dimension 1 and hence $\cont_S$ sends $G$ to points. If $\cont_S$ is a divisorial contraction then we are in case (b). 

If $\cont_S$ is a small contraction, we consider the associated flip $X\dashrightarrow X'$ and we denote by $D'$ the strict transform of $D$ in $X'$ and by $S'$ the corresponding small extremal ray on $X'$ with exceptional locus $G'\subseteq X'$ and satisfying $K_{X'}\cdot S' > 0$. Since $X$ is Fano and $G \cap D = \emptyset$ we have that $X'$ has an extremal ray $T\subseteq \overline{\NE}(X')$ such that $-K_{X'}\cdot T > 0$ and $D'\cdot T > 0$, from which we deduce that if we denote $\cont_T=:\psi:X'\to Y'$ then $T\not\subseteq \N_1(D',X')$ and hence $\psi$ is finite both on $D'$ and $G'$. In particular, every non trivial fiber of $\psi$ has dimension 1, since $D'\cdot T>0$. It only remains to prove that $\psi$ is of fiber type. Suppose, to the contrary, that $\psi$ is a birational contraction. If we suppose that $\Exc(\psi)\cap G \neq \emptyset$ and we consider $F_0$ to be an irreducible component of a fiber of $\psi$ that intersects $G'$ then it follows from the fact that $X$ is Fano Gorenstein that $-K_{X'}\cdot [F_0] > 1$, by \cite[Lemma 3.2]{DN14} (which is the singular version of \cite[Lemma 3.8]{Cas09}, used in step (3.3.5) of the proof of \cite[Proposition 3.3]{CD12}). On the other hand, we notice that $F_0 \not\subseteq G'$ since $\psi$ is finite on $G'$ and hence \cite[Lemma 1.1]{Ish91} can be applied (since $F_0$ contains a Gorenstein point of $X'$) in order to deduce that $-K_{X'}\cdot [F_0] \leq 1$, a contradiction. We obtain therefore that $\Exc(\psi)\cap G' = \emptyset$, so that $\Exc(\psi)$ is contained in the Gorenstein locus of $X'$ and therefore \cite[Theorem 2.2]{DN14} implies that $\Exc(\psi)$ is a divisor. But in this case we would have that $\Exc(\psi)\cdot S'=0$ and $\Exc(\psi)\cdot T < 0$ implying that $-\Exc(\psi)$ is nef, a contradiction. We conclude therefore that $\psi$ is of fiber type.
\end{proof}

\section{The toric MMP}\label{section:toricmmp}

We will analyze now the toric case. We may refer the reader to \cite{CLS} for the general theory of toric varieties and to \cite{Mat02} for details of the toric MMP. We will keep the same notation as \cite{CLS}.

Let $N\cong \mathbb{Z}^n$ be a lattice, $M=\operatorname{Hom}_\mathbb{Z}(N,\mathbb{Z})$ its dual lattice and let $N_\mathbb{R}$ (resp. $N_\mathbb{Q}$) and $M_\mathbb{R}$ (resp. $M_\mathbb{Q}$) be their real scalar (resp. rational scalar) extensions. Let us denote by $\langle \cdot,\cdot \rangle: M_\mathbb{R}\times N_\mathbb{R}\to \mathbb{R}$ the natural $\mathbb{R}-$bilinear pairing.

Let $\Delta_X\subseteq N_\mathbb{R}\cong \mathbb{R}^n$ be a fan. As we will see, most of the properties of our interest in the context of the MMP of the associated toric variety $X=X(\Delta_X)$, can be translated into combinatorial properties of the fan $\Delta_X$. 

Sometimes we will write $N_X$ instead of $N$ in order to emphasize the dependence of $X(\Delta_X)$ on the lattice where primitive generators of $\Delta_X$ belong.
 
Let $\Delta_X(k)$ be the set of $k$ dimensional cones in $\Delta_X$. In the same way, if $\sigma\in \Delta_X$ is a cone, we will denote by $\sigma(k)$ the set of its $k-$dimensional faces. Usually, we will not distinguish between $1-$dimensional cones $\rho\in \Delta_X(1)$ (or $1-$dimensional faces $\rho\in \sigma(1)$) and the primitive vector $u_\rho\in N$ generating them.

If $\sigma\in \Delta_X(k)$ we will denote by $U_\sigma$ the associated affine toric variety, and by $V(\sigma)\subseteq X(\Delta_X)$ the closed invariant subvariety of codimension $k$. In particular, each $\rho\in \Delta_X(1)$ corresponds to an invariant Weil divisor $V(\rho)$ on $X$ (also noted $V(u_\rho)$); such a cone is called a {\it ray}. Similarly, each cone of codimension 1 $\omega\in \Delta_X(n-1)$ corresponds to an invariant rational curve on $X$; such a cone is called a {\it wall}. 

All affine toric varieties associated to strongly convex rational polyhedral cones are normal (see \cite[Theorem 1.3.5]{CLS}). Thus, a toric variety $X$ associated to a fan $\Delta_X$ is also normal. Moreover, if we say that a cone $\sigma \in \Delta_X$ is smooth if and only if the associated affine toric variety $U_\sigma$ is smooth, then we have the following result.
\begin{prop}[{\cite[Proposition 11.1.2, Proposition 11.1.8]{CLS}}]\label{toricsing0} 
Let $X$ be the toric variety associated to the fan $\Delta_X$. Then,
$$\operatorname{Sing}(X)=\bigcup_{\sigma \text{ not smooth}}V(\sigma) $$
and
$$X_{\lisse}=\bigcup_{\sigma \text{ smooth}}U_\sigma. $$
Moreover, given a $d-$dimensional simplicial cone $\sigma\subseteq N_\mathbb{R}$ with generators $u_1,\ldots,u_d\in N$, let $N_\sigma = \operatorname{Span}(\sigma)\cap N$ and define the multiplicity of $\sigma$ by
$$ \mult(\sigma)=[N_\sigma:\mathbb{Z}u_1+\cdots+\mathbb{Z}u_d]. $$
Then,
\begin{enumerate}
\item $\sigma$ is smooth if and only if $\mult(\sigma)=1$.
\item Let $e_1,\ldots,e_d$ be a basis of $N_\sigma$ and write $u_i=\sum_{i=1}^d a_{ij}e_j$. Then,
$$\mult(\sigma)=\left|\det(a_{ij})\right|. $$
\item If $\tau \preceq \sigma$ is a face of $\sigma$, then
$$\mult(\sigma)=\mult(\tau)[N_\sigma:N_\tau+\mathbb{Z}u_1+\cdots+\mathbb{Z}u_d]. $$
In particular, $\mult(\tau)\leq \mult(\sigma)$ whenever $\tau \preceq \sigma$.
\end{enumerate}
\end{prop}

\begin{rema}\label{smoothincodimk}
Let $X$ be a $n-$dimensional toric variety associated to a simplicial fan $\Delta_X$, i.e., a fan whose cones are all simplicial, on which there is a cone $\sigma$ of full dimension $n$. If $X$ is smooth in codimension $k$, namely the closed invariant subset $\operatorname{Sing}(X)\subseteq X$ is such that $\codim_X \operatorname{Sing}(X) \geq k+1$, then we can choose a basis of $\mathbb{Z}^n\cong N$ in such a way the first $k$ generators of the cone $\sigma$ corresponds to the first $k$ elements of the canonical basis of $\mathbb{Z}^n$.
\end{rema}

In general, most of the interesting kind of singularities can also be characterized by looking at the (maximal) cones belonging to the fan.

\begin{theo}\label{toricsing1}
Let $\sigma$ be a strongly convex rational polyhedral cone and let $U_\sigma$ be the corresponding affine toric variety, then the following hold.
\begin{enumerate}
\item $U_\sigma$ is Cohen-Macaulay.
\item $U_\sigma$ is $\mathbb{Q}$-factorial if and only if $\sigma$ is simplicial.
\item $U_\sigma$ is $\mathbb{Q}$-Gorenstein if and only if there exists $m_\sigma\in M_\mathbb{Q}$ such that $\langle m_\sigma, u_\rho \rangle = 1$, for every ray $\rho\in \sigma(1)$. In this case, the Gorenstein index of $U_\sigma$ is the smallest positive integer $\ell\in \mathbb{Z}_{>0}$ such that $\ell m_\sigma \in M$.
\item If $U_\sigma$ is $\mathbb{Q}$-Gorenstein then $U_\sigma$ has log terminal singularities.
\item If $U_\sigma$ is $\mathbb{Q}$-Gorenstein then $U_\sigma$ has terminal singularities of Gorenstein index $\ell$ if and only if there exists $m_\sigma \in M$ such that
$$\langle m_\sigma, u_\rho \rangle = \ell \text{ for all }u_\rho\in \operatorname{Gen}(\sigma) \text{ and} $$
$$\langle m_\sigma, u_\rho \rangle > \ell \text{ for all }u_\rho\in \sigma\cap N \setminus (\{0\}\cup \operatorname{Gen}(\sigma)). $$
The element $m_\sigma$ is uniquely determined whenever $\sigma$ is of maximal dimension in the fan.
\item If $U_\sigma$ is $\mathbb{Q}$-Gorenstein then $U_\sigma$ has canonical singularities of Gorenstein index $\ell$ if and only if there exists $m_\sigma \in M$ such that
$$\langle m_\sigma, u_\rho \rangle = \ell \text{ for all }u_\rho\in \operatorname{Gen}(\sigma) \text{ and} $$
$$\langle m_\sigma, u_\rho \rangle \geq \ell \text{ for all }u_\rho\in \sigma\cap N \setminus (\{0\}\cup \operatorname{Gen}(\sigma)). $$
The element $m_\sigma$ is uniquely determined whenever $\sigma$ is of maximal dimension in the fan.
\item If $U_\sigma$ is Gorenstein then $U_\sigma$ has canonical singularities.
\end{enumerate}
Here, $\operatorname{Gen}(\sigma)$ denotes the set of primitive vectors $u_\rho$ generating all the rays $\rho=\{\lambda u_\rho\;|\;\lambda \geq 0\}\in \sigma(1)$
\end{theo}

\begin{proof} We may refer the reader to the survey \cite{Dai02} for proofs or references to proofs.
\end{proof}

\begin{rema}
By Theorem \ref{toricsing1} above, if $X$ is a $\mathbb{Q}$-factorial toric variety with canonical singularities then we have the decomposition
$$ \operatorname{Sing}(X)=\mathop{\bigcup_{\text{$\sigma$ canonical}}}_{\text{non-terminal}}  V(\sigma) \cup \mathop{\bigcup_{\sigma \text{ terminal}}}_{\text{non-smooth}} V(\sigma).$$
Therefore, if $X$ is a $\mathbb{Q}$-factorial toric variety with canonical singularities and with at most finitely many non-terminal points, then the (finite) set of canonical points is made up by some invariant points $V(\sigma)$, where $\sigma$ are of maximal dimension in the fan.
\end{rema}

If $X$ is a $\mathbb{Q}$-factorial complete toric variety of dimension $n$ then every extremal ray $R\subseteq \overline{\NE}(X)$ of $X$ corresponds to an invariant curve $C_\omega$ such that $R=\mathbb{R}_{\geq 0}[C_\omega]$ or, equivalently, to a wall $\omega\in \Delta_X(n-1)$. 

Let us suppose that $\omega=\operatorname{cone}(u_1,\ldots,u_{n-1})$, where $u_i$ are primitive vectors. Since $\Delta_X$ is a simplicial fan, $\omega$ separates two maximal cones $\sigma=\operatorname{cone}(u_1,\ldots,u_{n-1},u_n)$ and $\sigma'=\operatorname{cone}(u_1,\ldots,u_{n-1},u_{n+1})$, where $u_n$ and $u_{n+1}$ are primitive on rays on opposite sides of $\omega$. The $n+1$ vectors $u_1,\ldots,u_{n+1}$ are linearly dependent. Hence, they satisfy a so called {\it wall relation}:
$$b_n u_n + \sum_{i=1}^{n-1}b_iu_i+b_{n+1} u_{n+1}=0, $$ 
where $b_n,b_{n+1}\in \mathbb{Z}_{>0}$ and $b_i\in \mathbb{Z}$ for $i=1,\ldots,n-1$. By reordering if necessary, we can assume that
\begin{center}
$\begin{array}{ccc}
b_i<0 & \text{for} & 1\leq i \leq \alpha \\
b_i=0 & \text{for} & \alpha+1\leq i \leq \beta \\
b_i>0 & \text{for} & \beta+1\leq i \leq n+1.
\end{array}$ 
\end{center}
This wall relation and the signs of the coefficients involved allow us to describe the nature of the associated contraction. Mori precisely, one of the following cases occurs:
\begin{itemize}
 \item[(1)] Fiber type contraction: $\alpha=0$
 \item[(2)] Divisorial contraction: $\alpha=1$
 \item[(3)] Small contraction: $\alpha>1$.
\end{itemize}
In all case we have that $\dim \Locus(R) = n - \alpha$ and $\dim \cont_R(\Locus(R))=\beta$. We refer the reader to \cite[$\S 2$]{Re83} or \cite[$\S 14.2$]{Mat02} for details.\\

In general, if $\omega\in \Delta_X(n-1)$ is any wall of $\Delta_X$ (not necessarily corresponding to an extremal ray), we will also have a wall relation allowing us to compute the intersection number of the curve $C_\omega$ with every invariant divisor $V(\rho)$, $\rho\in \Delta_X(1)$. See \cite[Proposition 6.4.4]{CLS} for details.\\

In the setting of toric varieties, we will be interested in analyzing the extremal contractions appearing in \textsection \ref{section:extremalcontractions} in terms of the combinatorial given by \cite[Theorem 2.4, Corollary 2.5]{Re83}. Let us begin by the birational case.

\begin{lemm}\label{toricdivisorial}
Let $X=X(\Delta_X)$ be a $\mathbb{Q}$-factorial Gorenstein toric variety of dimension $n\geq 3$. Let $\varphi_R:X\to X_R$ be a divisorial contraction with exceptional divisor $\Exc(\varphi_R)=E$ such that
\begin{enumerate}
\item $A=\varphi_R(E)$ is an invariant subvariety of codimension two.
\item $E\cdot [F]=K_X\cdot [F]=-1$ for every non-trivial fiber $F$ of $\varphi_R$.
\end{enumerate}
Let us suppose that $\varphi_R:X\to X_R$ is defined by the contraction of the wall $\omega=\operatorname{cone}(u_1,\ldots,u_{n-1})$ that separates the maximal cones $\sigma=\operatorname{cone}(u_1,\ldots,u_{n-1},u_n)$ and $\sigma'=\operatorname{cone}(u_1,\ldots,u_{n-1},u_{n+1})$. Then, up to reordering if necessary, the wall relation satisfied by these cones is of the form
$$ u_n+u_{n+1}=u_1. $$
If moreover $X$ has isolated singularities, then $X_R$ has isolated Gorenstein singularities, $A$ is contained in the smooth locus of $X_R$, and $X \cong \operatorname{Bl}_A(X_R)$.
\end{lemm}

\begin{proof}
By \cite[Theorem 2.4, Corollary 2.5]{Re83}, up to reordering if necessary, we can suppose that $\varphi_R:X\to X_R$ is defined by the relation
\begin{equation}\alpha u_n + \lambda u_1 + \beta u_{n+1}=0, \end{equation}
where $\alpha,\beta \in \mathbb{Z}_{>0}$, $\lambda \in \mathbb{Z}_{<0}$ and $E=V(u_1)$. Let us denote by $C=V(\omega)$ the invariant curve associated to the wall $\omega$.

By \cite[Proposition 6.4.4]{CLS}, $V(u)\cdot [C] = 0$ for $u\not \in \{u_1,\ldots,u_{n+1}\}$. Moreover, the wall relation gives us $V(u_i)\cdot [C] = 0$ for $i\in \{2,\ldots,n-1\}$ and 
$$V(u_1)\cdot [C] = \dfrac{\lambda}{\alpha}V(u_n)\cdot [C] = \dfrac{\lambda}{\beta}V(u_{n+1})\cdot [C]. $$
By hypothesis, $V(u_1)\cdot [C] = -1$ and thus $V(u_n)\cdot [C] = -\frac{\alpha}{\lambda}$ and $V(u_{n+1})\cdot [C] = -\frac{\beta}{\lambda}$.

It is well know that for a toric variety $X=X(\Delta_X)$ we have that $K_X=-\sum_{\rho\in \Delta_X(1)}V(\rho)$ is an invariant canonical divisor on $X$ (see \cite[Theorem 8.2.3]{CLS}), and thus the condition $-K_X\cdot [C] = 1$ can be translated into
\begin{equation}-1+\left(-\dfrac{\alpha}{\lambda}\right)+\left(-\dfrac{\beta}{\lambda}\right)=1 \Leftrightarrow \alpha+\beta = -2\lambda. \end{equation}
On the other hand, we should notice that we can suppose that $\gcd(\lambda,\alpha)=\gcd(\lambda,\beta)=\gcd(\alpha,\beta)=1$. In fact, if $\gcd(\alpha,\beta)=d>1$, then the equation (1) implies $d|\lambda$, as $u_1$ is a primitive vector. The same argument applies to the other two pairs.

By assumption, $K_X=-\sum_{\rho\in\Delta_X(1)}V(\rho)$ is a Cartier divisor, i.e., 
for each maximal cone $\sigma\in \Delta_X(n)$, there is $m_\sigma\in M$ with $\langle m_\sigma,u_\rho\rangle =1 \text{ for all }\rho\in \sigma(1)$. In our setting, this condition applied to the two maximal cones $\sigma$ and $\sigma'$ tells us that there exist two elements $m,m'\in M$ such that
$$\langle m,u_i\rangle = 1 \qquad\text{for}\qquad i\in\{1,\ldots,n\}, $$
$$\langle m',u_i\rangle = 1 \qquad\text{for}\qquad i\in\{1,\ldots,n-1,n+1\}. $$
From the equation (1) we obtain 
\begin{equation}
\alpha+\lambda+\beta\langle m,u_{n+1}\rangle = 0, \text{ and} 
\end{equation}
\begin{equation}
\alpha \langle m',u_n \rangle + \lambda + \beta = 0
\end{equation}
By using the equation (2) and (3) we obtain that $\lambda = \beta (\langle m,u_{n+1}\rangle - 1)$ and thus $\beta=1$, as $\beta\in \mathbb{Z}_{>0}$ and $\gcd(\lambda,\beta)=1$. In the same way, by using the equation (2) and (4), we deduce that $\alpha=1$ and hence $\lambda=-1$. Finally, we get the relation
$$u_n + u_{n+1}=u_1,$$
and $\langle m,u_{n+1} \rangle = \langle m',u_n \rangle = 0$.

Let us suppose now that $X$ has isolated singularities. We note that there are exactly $n-1$ walls satisfying this relation (corresponding to the fibers over $n-1$ invariant points of $A)$, $\omega_i$ with $i=1,\ldots,n-1$, each of them separating two maximal cones $\sigma_i$ and $\sigma_i'$. It follows from the wall relation above and the intersection theory for toric varieties (see \cite[Proposition 6.4.4]{CLS}) that 
 $$E \cdot [C_{\omega_i}] = -\dfrac{\mult(\omega_i)}{\mult(\sigma_i)}=-\dfrac{\mult(\omega_i)}{\mult(\sigma_i')}=-1, $$
 and
 $$-K_X\cdot [C_{\omega_i}]=\dfrac{\mult(\omega_i)}{\mult(\sigma_i)}=\dfrac{\mult(\omega_i)}{\mult(\sigma_i')}=1. $$
Therefore, $\mult(\omega_i)=\mult(\sigma_i)=\mult(\sigma_i')$. 

On the other hand, Remark \ref{smoothincodimk} implies that $\mult(\omega_i)=1$ since $X$ has isolated singularities. Hence, both $\sigma_i$ and $\sigma_i'$ are smooth cones in $\Delta_X(n)$. We get that the associated maximal cone $\tau_i\in \Delta_{X_R}(n)$ obtained from $\sigma_i$ and $\sigma_i'$ by removing the ray corresponding to the exceptional divisor $E=\operatorname{Exc}(\varphi_R)$ is also smooth, for $i=1,\ldots,n-1$. Thus, $A=V(u_n,u_{n+1})$ is contained in the smooth locus of $X_R$. The isomorphism $X\cong \operatorname{Bl}_A(X_R)$ follows from the description of the blow-up of a smooth toric variety along an irreducible invariant smooth subvariety (see \cite[Definition 3.3.17]{CLS}). A posteriori, we note that $X_R$ also have isolated Gorenstein singularities since $A$ is contained in the smooth locus of $X_R$.
\end{proof}

Let us consider now the case when $\varphi_R:X\to X_R$ is a contraction of fiber type.

\begin{rema}[Flatness of toric fibrations] \label{rema: flatness}
Contrary to what was originally written in \cite[Corollary 2.5]{Re83}, the restriction of an extremal contraction to its exceptional locus may not be flat in general, as Miles Reid kindly communicated to the author. See \cite[Example 13]{ArRM} for a local counter-example that can be compactified by \cite{Fuj06}. We refer the interested reader to Kato's criterion of flatness for toric morphisms \cite[Proposition 4.1]{Kat89}.
\end{rema}

We will need the following result concerning extremal contractions of fiber type.

\begin{lemm} \label{toricfibration}
Let $X$ be a $\mathbb{Q}$-factorial Gorenstein toric variety of dimension $n\geq 3$ and let $\varphi_R:X\to X_R$ be a proper toric morphism given by the contraction of an extremal ray $R\subseteq \overline{\NE}(X)$ of fiber type. Suppose that all the fibers of $\varphi_R$ are of dimension 1. If $\varphi_R$ has a non-reduced fiber then $\Sing(X)$ has an irreducible component of codimension 2. In particular, if $X$ has at most finitely many non-terminal points then $\varphi_R$ is a $\mathbb{P}^1$-bundle and there exists a split vector bundle $E$ of rank $2$ on $X_R$ and an isomorphism $X\cong \mathbb{P}_{X_R}(E)$ over $X_R$.
\end{lemm}

\begin{proof}
The statement is local on $X_R$ so we may suppose that $X_R=U_\tau$ is an affine simplicial toric variety. If follows from \cite[Theorem 2.4, Corollary 2.5]{Re83} that if $\varphi_R:X\to X_R$ is an extremal contraction of pure relative dimension one then all the fibers of $\varphi_R$ are irreducible, the fan of $X=\varphi_R^{-1}(U_\tau)$ has two maximal cones $\sigma=\operatorname{cone}(u_1,\ldots,u_n)$ and $\sigma'=\operatorname{cone}(u_0,\ldots,u_{n-1})$ such that $u_0=-u_n$, and $\varphi_R$ is induced by the projection ${\Phi:N_{X} \to N_{X_R} = N_{X}/(\operatorname{Span}(u_n)\cap N_X)}$ onto $\tau = \operatorname{cone}(t_1,\ldots,t_{n-1})$. Since $X$ is smooth in codimension one, we can always suppose that $u_n=e_n$ and hence that $\varphi_R$ is induced by the projection $\Phi(x_1,\ldots,x_n)=(x_1,\ldots,x_{n-1},0)$.
 
 It follows from \cite[Chapter 2, Lemma 5.2]{Karu99} that if $\varphi_R$ has a non-reduced fiber over $U_\tau$ then there exists $u_i$ such that $\Phi(u_i)$ is not primitive on $N_{X_R}$ for some $i=1,\ldots,n-1$. Let us suppose therefore that $\Phi(u_i)=\lambda_i t_i$ with $t_i \in \mathbb{Z}^{n-1}\times \{0 \}$ primitive lattice vector and $\lambda_i \geq 2$. Let us suppose that $u_i=(a_1,\ldots,a_n)\in \mathbb{Z}^{n}$ and hence that $\Phi(u_i)=(a_1,\ldots,a_{n-1},0)=\lambda_i(\overline{a}_1,\ldots,\overline{a}_{n-1},0)$ with $t_i=(\overline{a}_1,\ldots,\overline{a}_{n-1},0)$ primitive lattice vector.
 
 Since $X$ is Gorenstein there are integer (dual) vectors $m_\sigma=(m_1,\ldots,m_n)\in \mathbb{Z}^n$ and $m_\sigma'=(m_1',\ldots,m_n')\in \mathbb{Z}^n$ such that $\langle m_\sigma,u_i \rangle = 1$ for $i=1,\ldots,n$ and $\langle m_\sigma',u_i \rangle = 1$ for $i=0,\ldots,n-1$, by Theorem \ref{toricsing1}. In particular we have that $m_n=1$, $m_n'=-1$ and 
 \begin{equation}
  \langle m_\sigma, u_i \rangle = m_1a_1+\ldots+m_{n-1}a_{n-1}+a_n=1,
 \end{equation}
 \begin{equation}
  \langle m_\sigma', u_i \rangle = m_1'a_1+\ldots+m_{n-1}'a_{n-1}-a_n=1.
 \end{equation}
By adding equation (1) and (2) above, we get that $\lambda_i \langle m_\sigma+m_\sigma', t_i \rangle = 2$ and hence $\lambda_i = 2$. We can write therefore $u_i=(2\overline{a}_1,\ldots,2\overline{a}_{n-1},a_n)\in \mathbb{Z}^n$ with $a_n$ odd number and $t_i=(\overline{a}_1,\ldots,\overline{a}_{n-1},0)$.

Let us produce now a lattice vector $\overline{u}_i \in \mathbb{Z}^n$ lying over $t_i$ and belonging to the hyperplane $\{u\in \mathbb{R}^n\;|\; \langle m_\sigma,u\rangle = 1\}$. We have that
\begin{equation*}
 \begin{aligned}
 1 & = m_1a_1+\ldots+m_{n-1}a_{n-1}+a_n \\
   & = m_1 2\overline{a}_1 + \ldots + m_{n-1} 2 \overline{a}_{n-1} + a_n \\
   & = m_1 \overline{a}_1 + \ldots + m_{n-1} \overline{a}_{n-1} + b = \langle m_\sigma,\overline{u}_i \rangle,
 \end{aligned}
\end{equation*}
where $\overline{u}_i=(\overline{a}_1,\ldots,\overline{a}_{n-1},b)$ and $b=\overline{a}_1 + \ldots + m_{n-1} \overline{a}_{n-1}+a_n\in \mathbb{Z}$. We note that $2\overline{u}_i = u_i+u_n$ and thus $\overline{u}_i \in \operatorname{cone}(u_i,u_n)$. We conclude from Proposition \ref{toricsing0} and Theorem \ref{toricsing1} that $X$ has canonical singularities along the codimension two closed subvariety $V(u_i,u_n)$, which are therefore non terminal by codimension reasons.

Finally, if $X$ has at most finitely many non-terminal points we conclude therefore that all the fibers of $\varphi_R$ are irreducible and reduced and thus $\varphi_R$ is a $\mathbb{P}^1$-bundle. Moreover, the invariant divisors $V(u_0)$ and $V(u_n)$ correspond to two disjoint invariant sections $s_0:X_R\to X$ and $s_\infty:X_R\to X$ passing through the two invariant points of all the fibers of $\varphi$. Hence, \cite[Remark 8]{ArRM} implies that there exists a rank 2 split vector bundle $E$ such that $X\cong \mathbb{P}_{X_R}(E)$. 
\end{proof}

\section{The extremal case $\rho_X=3$ for toric varieties}\label{section:rho=3toric}

We are now able to prove the structure theorem for toric varieties with Picard number 3.

\begin{proof}[Proof of Theorem \ref{thm:rho=3toric}]

By Theorem \ref{thm:rho=3}, we obtain a diagram
$$\xymatrix{X \ar[r]^{\sigma} \ar@/_1pc/[rr]_{\varphi} & Y \ar[r]^{\pi} & Z}, $$
where $\sigma:X\to Y$ is a divisorial contraction sending a toric prime divisor $E$ onto an invariant subvariety $A$ of codimension two on $Y$, and $\pi$ and $\varphi$ are both extremal contractions of fiber type whose fibers are of dimension 1. All these varieties are $\mathbb{Q}$-factorial toric Fano varieties.

Let us first prove that $\pi:Y \to Z$ is a $\mathbb{P}^1-$bundle. It follows from \cite[Proposition II. 1.1]{Gr71} that $S_\pi$, the locus of points of $Z$ over which $\pi$ is not a smooth morphism, is a closed subset of $Z$. Moreover, as we noticed in Remark \ref{fibersP1}, we have that if $X$ is a Fano variety satisfying $(\dag)$ then $\pi$ and $\varphi$ are {\it generalized conic bundles} and therefore the 1-cycles associated to singular fibers of $\pi$ are of the form $[F]=[C]+[C']$ with $C$ and $C'$ (eventually coincident) irreducible and generically reduced rational curves on $Y$ such that $C_{\redu} \cong C'_{\redu} \cong \mathbb{P}^1$ and $-K_X \cdot [\widetilde{C}] = -K_X \cdot [\widetilde{C}']= -K_Y \cdot [C] = -K_Y \cdot [C']=1$, where $\widetilde{C}$ (resp. $\widetilde{C}'$) is the total transform of $C$ (resp. $C'$) on $X$ via $\sigma$. In particular, we observe as in the Proof of Corollary \ref{thm:threefold} that the singular fibers of $\pi$ must be disjoint from $A\subseteq Y$ and hence they are contained in the Gorenstein locus of $Y$. We conclude from 
Lemma \ref{toricfibration} that $\pi:Y \to Z$ is a $\mathbb{P}^1-$bundle isomorphic to the projectivization of the rank 2 split vector bundle $E=\mathcal{L}'\oplus \mathcal{L}$.

As $\mathbb{P}_Z(E)\cong \mathbb{P}_Z(E\otimes \mathcal{M})$ for any line bundle $\mathcal{M}\in \Pic(Z)$, we can suppose that $\mathcal{L}'\cong \mathcal{O}_Z$. On the other hand, since $\Pic(Z)$ is isomorphic to $\mathbb{Z}$, we can consider an ample generator $\mathcal{O}_Z(1)$ of $\Pic(Z)$ and an integer $a\in \mathbb{Z}$ such that $\mathcal{L}\cong \mathcal{O}_Z(a)$. Up to tensor by $\mathcal{L}^{\vee}$, we can always suppose that $a\geq 0$. In particular, both $Y$ and $Z$ must have at most terminal singularities, since $Y$ has a most a finite number of canonical singularities and $\pi$ is locally trivial. Moreover, $0\leq a \leq i_Z -1$ by Proposition \ref{thm:rho=2 toric fibration} below.

It should be noticed that in this case $A\subseteq Y$ is contained in one of the two disjoint invariant sections associated to the $\mathbb{P}^1-$bundle $\pi:Y\to Z$. In fact, by Theorem \cite[Theorem 2.4, Corollary 2.5]{Re83} (taking $\alpha=0$ and $\beta=n-1$) we have that $\pi:Y \to Z$ is defined by a wall relation of the form
$$b_n u_n +\sum_{i=1}^{n-1}0\cdot u_i + b_{n+1} u_{n+1}=0, $$
where $\omega=\operatorname{cone}(u_1,\ldots,u_{n-1})$ is a wall generating the extremal ray associated to this contraction, which separates the two maximal cones $\sigma=\operatorname{cone}(u_1,\ldots,u_{n-1},u_n)$ and $\sigma'=\operatorname{cone}(u_1,\ldots,u_{n-1},u_{n+1})$.

Thus $U(\omega)=\operatorname{cone}(u_n,u_{n+1})=\operatorname{Span}(u_n)=\operatorname{Span}(u_{n+1})\cong \mathbb{R}$ and $\pi:Y \to Z$ is induced by the quotient $N_\mathbb{R}\to N_\mathbb{R}/U(\omega)$. Since $A\subseteq D_Y=\sigma(D)$ and $\pi|_{D_Y}$ is a finite morphism, $A$ is sent onto an invariant subvariety of dimension $n-2$, a divisor on $Z$. Thus, $A\subseteq D(u_n):=D_0$ or $A\subseteq D(u_{n+1}):=D_\infty$, otherwise it will be sent onto a subvariety of dimension $n-3$.
 
If we suppose that $A\subseteq D_\infty$, where $D_\infty \cong Z \subseteq Y$ is one of these disjoint invariant sections, then for any invariant affine open subset $U\subseteq Z$ such that $\pi^{-1}(U)\cong U \times \mathbb{P}^1$ we will have
$$\pi^{-1}(U) \setminus (\pi^{-1}(U)\cap D_\infty) \cong U \times \mathbb{A}^1.  $$
The open set $ U \times \mathbb{A}^1$ is therefore isomorphic to an open set on $Y$ contained in the locus where $\sigma^{-1}:Y\dashrightarrow X$ is an isomorphism. In particular, $U\times \mathbb{A}^1$ is an affine Gorenstein toric variety and thus $U$ is also Gorenstein. We conclude in this way that both $Y$ and $Z$ are Gorenstein varieties.

As a consequence of the formula $K_X = \sigma^*(K_Y)+ E$ we have that $E=\Exc(\sigma)$ is a Cartier divisor. Let us prove that $X\to Y$ verifies the universal property of the blow-up. The short exact sequence of sheaves
$$0 \to \mathcal{O}_X(-E) \to \mathcal{O}_X \to \mathcal{O}_E \to 0 $$
gives
$$0\to \sigma_* \mathcal{O}_X(-E)\to \sigma_* \mathcal{O}_X \to (\sigma|_E)_*\mathcal{O}_E \to \operatorname{R}^1 \sigma_* \mathcal{O}_X(-E)\to \cdots, $$
where $\sigma_* \mathcal{O}_X = \mathcal{O}_Y$ since $\sigma:X\to Y$ is a contraction.

Notice that $\sigma|_E:E\to A$ is a $\mathbb{P}^1-$bundle. In fact, since $K_X$ is a Cartier divisor and $-K_X\cdot [F]=1$ for any non-trivial fiber $F$ of $\sigma$, it follows that the scheme theoretic fiber $F$ is an irreducible and generically reduced rational curve on $X$. Then, by \cite[Theorem II.2.8]{Kol96}, $\sigma|_E:E\to A$ is a $\mathbb{P}^1-$bundle and thus $(\sigma|_E)_*\mathcal{O}_E = \mathcal{O}_A$. 
 
On the other hand, the Cartier divisor $-(K_X+E)$ is $\sigma$-ample and therefore $\operatorname{R}^i \sigma_* \mathcal{O}_X(-E)=0$ for $i>0$, by \cite[Vanishing Theorem 1.1]{AW95}. Hence, the above long exact sequence becomes
$$0 \to  \sigma_* \mathcal{O}_X(-E)\to \mathcal{O}_Y \to \mathcal{O}_A \to 0,$$
and thus $\mathcal{I}_A \cong \sigma_* \mathcal{O}_X(-E)$.

Let us follow \cite{AW93} and notice that $\sigma:X\to Y$ is a local contraction supported by the Cartier divisor $K_X-E$. Let $F$ be any non-trivial fiber of $\sigma$. Then, by \cite[Theorem 5.1]{AW93}, the evaluation morphism
$$\sigma^* \sigma_* \mathcal{O}_X(-E) \to \mathcal{O}_X(-E) $$
is surjective at every point of $F$. On the other hand, $\sigma_* \mathcal{O}_X(-E)\cong \mathcal{I}_A$ and $\sigma^{-1}\mathcal{I}_A \cdot \mathcal{O}_X$ is defined to be the image of $\sigma^*\mathcal{I}_A\to \mathcal{O}_X(-E)$. Thus, $\sigma^{-1}\mathcal{I}_A\cong \mathcal{O}_X(-E)$ is an invertible sheaf. 

Then, by the universal property of the normalized blow-up, $\sigma$ factorizes as
$$\xymatrix{X \ar[r]^{\tau \quad} \ar@/_1pc/[rr]_{\sigma} & \overline{\operatorname{Bl}_A(Y)} \ar[r]^{\quad \nu\circ\varepsilon} & Y}, $$
where $\varepsilon: \operatorname{Bl}_A(Y)\to Y$ is the blow-up of the coherent sheaf of ideals $\mathcal{I}_A$ and $\nu:\overline{\operatorname{Bl}_A(Y)}\to \operatorname{Bl}_A(Y)$ its normalization. 

Since $\sigma$ contracts only the irreducible divisor $E$, $\tau$ contracts no divisor. If $\tau$ is not finite, it is a small contraction sending a curve $C\subseteq E$ to a point. The rigidity lemma \cite[Lemma 1.6]{KM98} applied to the $\mathbb{P}^1$-bundle $E\to A$ and the morphism $\tau(E)\to A$ implies that $\tau$ contracts the divisor $E$, a contradiction. Hence $\tau$ is a finite and birational morphism onto a normal variety, and therefore $\tau$ is an isomorphism by Zariski's Main Theorem.

Finally, it follows from \cite[Proposition 11.4.22]{CLS} that if $X=X(\Delta_X)$ is a $\mathbb{Q}$-factorial toric variety with Gorenstein terminal singularities, then $\operatorname{codim}_X \operatorname{Sing}(X)\geq 4$. Therefore, if we start with a toric variety $X$ of dimension 3 or 4 satisfying the hypothesis of Theorem \ref{thm:rho=3toric}, we will obtain that $Z$ is a smooth toric variety, implying that $Y$ and $X$ must be both smooth toric varieties too. 
\end{proof}

\section{The case $\rho_X=2$ for toric varieties}\label{section:rho=2toric}

In this section we study the extremal contractions described in \textsection \ref{section:extremalcontractions} for toric varieties with Picard number 2. Let us begin with the proof of Proposition \ref{thm:rho=2 toric fibration}.

\begin{proof}[Proof of Proposition \ref{thm:rho=2 toric fibration}]
The isomorphism $X\cong \mathbb{P}_Y(\mathcal{O}_Y\oplus \mathcal{O}_Y(a))$, with $a\geq 0$, follows from Lemma \ref{toricfibration} and the fact that $\rho_Y=1$. On the other hand, the condition $0\leq a \leq i_Y$ is in fact equivalent to the condition of $X$ being Fano. To see this, let us recall that the adjunction formula for projectivized vector bundles \cite[\textsection 1.1.7]{BS95} gives
$$K_X = \pi^*\mathcal{O}_Y(a-i_Y)- 2 \xi,$$
where $\xi$ is the tautological divisor on $\mathbb{P}_Y(\mathcal{O}_Y\oplus \mathcal{O}_Y(a))$ and $i_Y$ is defined in such a way $\mathcal{O}_Y(-K_Y)\cong \mathcal{O}_Y(i_Y)$. Now, we notice that if $F$ is any fiber $\pi$, then
$$ -K_X\cdot [F] = 2\xi \cdot [F] =2 > 0.$$ 
On the other hand, if $C\subseteq Y$ is an irreducible reduced curve and $C_X \subseteq X$ is the image of $C$ by the section associated to the quotient $\mathcal{O}_Y\oplus \mathcal{O}_Y(a) \to \mathcal{O}_Y$, then
$$-K_X\cdot [C_X] = \deg(C)\cdot (i_Y-a). $$
Hence, $X$ is Fano if and only if $0 \leq a \leq i_Y-1$, since the numerical classes of these curves above generates the Mori cone of $X$. Finally, let us note that the two invariant sections of $\pi$, $s_0:Y\to X$ and $s_\infty: Y \to Z$, provide divisors $D_0=s_0(Y)$ and $D_\infty = s_\infty(Y)$ on $X$ such that $D_0\cong D_\infty \cong Y$ and hence $\dim_\mathbb{R} \N_1(D_0,X)=\dim_\mathbb{R} \N_1(D_\infty,X)= 1$.
\end{proof}

Let us prove now Theorem \ref{thm:rho=2 toric div isolated}.

\begin{proof}[Proof of Theorem \ref{thm:rho=2 toric div isolated}]

Let $R\subseteq \overline{\NE}(X)$ be an extremal ray such that $D\cdot R > 0$ and let $\pi:X\to Y$ be the corresponding extremal contraction. If $\pi$ is of fiber type then Proposition \ref{thm:rho=2 toric fibration} provides an isomorphism $X \cong \mathbb{P}_Y(\mathcal{O}_Y\oplus \mathcal{O}_Y(a))$ for some toric variety $Y$. Moreover, $Y$ is a $\mathbb{Q}$-factorial Gorenstein toric Fano variety of dimension $(n-1)$ with terminal singularities and Fano index $i_Y$, and $0\leq a \leq i_Y - 1$. Since $X$ has isolated singularities, and $\pi:X\to Y$ is a locally trivial $\mathbb{P}^1$-bundle, both $X$ and $Y$ are smooth in this case. It follows that $Y\cong \mathbb{P}^{n-1}$ and $i_Y = n$, which leads us to the first case.

Let us suppose that $\pi:X\to Y$ is a birational contraction. It follows from Proposition \ref{extcont} that $\pi:X\to Y$ is a divisorial contraction sending an irreducible invariant divisor $E=V(u_E)$ onto a codimension two subvariety $A\subseteq Y$, and $Y$ is a $\mathbb{Q}$-factorial toric Fano variety. Moreover, $E\cdot [F]=K_X \cdot [F]=-1$ for every non-trivial fiber $F$ of $\pi$. In particular, Lemma \ref{toricdivisorial} implies that $Y$ has isolated Gorenstein singularities, $A$ is contained in the smooth locus of $Y$, and $X \cong \operatorname{Bl}_A(Y)$. Additionally, it follows from the proof of Lemma \ref{toricdivisorial} that $Y$ contains $n-1$ smooth maximal cones, each of them containing the cone of dimension two defining the subvariety $A$. 

Since $Y$ is a complete and simplicial toric variety of dimension $n$ and Picard number one, the fan of $Y$ contains exactly $n+1$ maximal cones (corresponding to $n+1$ invariant points), $Y$ has at most 2 singular points and they are outside $A\subseteq Y$.

Let us denote by $u_1,\ldots,u_{n+1}\in N$ the primitive lattice vectors generating the rays in the fan of $Y$ and suppose that $A=\operatorname{cone}(u_1,u_2)$, and thus that the extremal contraction $\pi$ is defined by the wall relation $u_E= u_1+u_2$. There are exactly $n-1$ walls in $\Delta_X$ satisfying this relation (corresponding to the fibers over the $n-1$ invariants points of $A$). Namely, the walls
 $$\omega_i = \operatorname{cone}(u_E,u_3,\ldots,\widehat{u_i},\ldots,u_{n+1})\quad \text{with }i\in\{3,\ldots,n+1\}, $$
separating the two maximal cones $\sigma_i=\operatorname{cone}(u_E,u_3,\ldots,\widehat{u_i},\ldots,u_{n+1},u_1)$ and $\sigma_i'=\operatorname{cone}(u_E,u_3,\ldots,\widehat{u_i},\ldots,u_{n+1},u_{2})$.

Let us prove that $Y$ is isomorphic to one of the listed varieties. Since the fan of $Y$ contains a smooth maximal cone, we can suppose that the vectors $u_1,\ldots,u_n$ correspond to the first $n$ elements of the canonical basis of $\mathbb{Z}^n$, by Remark \ref{smoothincodimk}. 

Let us write $u_{n+1}=(-a_1,\ldots,-a_n)$, with $a_i\in \mathbb{Z}_{>0}$ for $i\in\{1,\ldots,n\}$. For each $i\in\{3,\ldots,n+1\}$ we have that $\mult (\sigma_i)=1$. Therefore, Proposition \ref{toricsing0} leads $ \mult(\sigma_i)=|\det(e_1,e_1+e_2,e_3,\ldots,\widehat{e_i},\ldots,e_n,u_{n+1})|=a_i=1 $, for $i\in \{3,\ldots,n\}$. Hence, we can write $u_{n+1}=(-a,-b,-1,\ldots,-1)$, with $a,b\in \mathbb{Z}_{>0}$. It should be noticed that $Y$ has isolated singularities if and only if $\gcd(a,b)=1$.

Thus, $Y\cong \mathbb{P}(1^{n-1},a,b)$ with $a,b\in \mathbb{Z}_{>0}$ relatively prime integers. Now, $Y$ is a Gorenstein Weighted Projective Space if and only if $a|(n-1+a+b)$ and $b|(n-1+a+b)$, by \cite[Lemma 3.5.6]{CK99}. Equivalently, $a|(n-1+b)$ and $b|(n-1+a)$. If $a=b$ the only possibility is $a=b=1$, leading to (a): $Y\cong \mathbb{P}^n$.

Let us suppose that $1\leq a < b$ and notice that $ab|(n-1+a+b)$ since $\gcd(a,b)=1$. On the other hand,
\begin{equation*}
\begin{aligned}
\dfrac{n-1+a+b}{ab}=1 & \Leftrightarrow \dfrac{n-1+a+b}{ab} < 2 \\
 & \Leftrightarrow (2b-1)\left(a-\frac{1}{2}\right)-n+\frac{1}{2}>0.
\end{aligned}
\end{equation*}
Since $a\geq 1$, this condition is fulfilled when $(2b-1)\cdot \frac{1}{2}-n+\frac{1}{2}>0\Leftrightarrow b>n$.

 Therefore, $b\geq n+1$ implies that $n-1+a+b=ab$ or, equivalently, $n=(a-1)(b-1)$. This leads to $a=2$, $b=n+1$ and hence to (b): $Y\cong \mathbb{P}(1^{n-1},2,n+1)$ and $n$ must be even. Finally, if $1\leq a < b \leq n$ we get the last case (c): $Y\cong \mathbb{P}(1^{n-1},a,b)$.

Conversely, given one of these listed varieties $Y$ with their fans as above and considering $X$ to be the blow-up of $Y$ along $A=V(e_1,e_2)$, we obtain a projective toric variety satisfying the hypothesis.

In fact, since $A$ is contained in the smooth locus of $Y$ we obtain that $X$ is a $\mathbb{Q}$-factorial Gorenstein toric variety with isolated canonical singularities. In order to prove that $X$ is Fano we need to analyze the second extremal contraction that corresponds to the wall relation on $X$ given by
$$ bu_E+(a-b)e_1+e_3+\cdots+e_n+u_{n+1}=0. $$
In any of the three listed cases we will obtain
$$-K_X\cdot [C_\omega] = \dfrac{n-1+a}{b}\in \mathbb{Z}_{>0}, $$
proving that $X$ is a Fano variety, by the Cone Theorem and Kleiman's criterion of ampleness. Finally, let us note that $A \subseteq Y$ is the intersection of two invariant prime divisors $D_1=V(u_1)$ and $D_2=V(u_2)$, whose strict transforms $\widehat{D}_1$ and $\widehat{D}_2$ in $X$ satisfy $\dim_\mathbb{R} \N_1(\widehat{D}_1,X)=\dim_\mathbb{R} \N_1(\widehat{D}_2,X)= 1$, which follows from the wall relation $u_E=u_1+u_2$ and the description of the fans defining these divisors as toric varieties (see for instance \cite[Proposition 3.2.7]{CLS}).
\end{proof}
 
As a consequence we obtain the following list of possible admissible weights $(a,b)\in \mathbb{Z}_{>0}^2$ that corresponds to varieties $Y\cong \mathbb{P}(1^{n-1},a,b)$ as in Theorem \ref{thm:rho=2 toric div isolated}.2. Compare with \cite{Kas13} and \cite{RM85}.

\begin{table}[H]
\centering
\caption{Admissible $Y\cong \mathbb{P}(1^{n-1},a,b)$ as in Theorem \ref{thm:rho=2 toric div isolated}.2}
\label{my-label}
\begin{tabular}{c|c}
$n$ & Weights $(a,b)\in \mathbb{Z}_{>0}^2$ \\ \hline
 3& $(1,1),(1,3)$ \\  \hline
 4& $(1,1),(1,2),(1,4),(2,5)$ \\  \hline
 5& $(1,1),(1,5)$ \\ \hline
 6& $(1,1),(1,3),(1,6),(2,7),(3,4)$ \\ \hline
 7& $(1,1),(1,7)$ \\ \hline
 8& $(1,1),(1,2),(1,4),(1,8),(2,3),(2,9),(3,5)$ \\ \hline
 9& $(1,1),(1,3),(1,9)$ \\ \hline
 10& $(1,1),(1,2),(1,5),(1,10),(2,11)$  \\ \hline
\end{tabular}
\end{table}

\begin{rema}

From the wall relation 
$$ bu_E+(a-b)e_1+e_3+\cdots+e_n+u_{n+1}=0 $$
we can deduce the nature of the second extremal contraction $\varphi:X\to W$. In the smooth case it is a contraction of fiber type onto $W \cong \mathbb{P}^1$. In the singular case, it is a divisorial contraction sending its exceptional locus onto a point. Moreover, $W\cong \mathbb{P}(1^{n-1},a,b-a)$ since
$$au_E+(b-a)e_2+e_3+\cdots +e_n+u_{n+1}=0.$$

\end{rema} 

We conclude with an example showing that is the hypothesis of isolated singularities in Theorem \ref{thm:rho=2 toric div isolated}.2 cannot be dropped. We exhibit an example of a $\mathbb{Q}$-factorial Gorenstein toric Fano fivefold $X$ with terminal singularities, whose singular locus is one-dimensional and that admits a birational extremal contraction $\pi:X\to Y$ 
which is not a blow-up, but only a blow-up in codimension two, and where $Y$ is a non-Gorenstein $\mathbb{Q}$-factorial toric Fano fivefold.

\begin{exem}\label{contreexample} Let us consider the fan $\Delta_X \subseteq \mathbb{R}^5$ generated by the vectors
$$e_1,e_2,e_3,e_4,e_5,u_6,u_E, $$
where $\{e_1,\ldots,e_5\}$ is the canonical basis of $\mathbb{R}^5$, $u_6=(-1,-1,-1,-2,-3)$ and $u_E=(-1,-1,-1,-2,-2)$. 

It can be checked by using \cite{Mac2} that $X$ is a $\mathbb{Q}$-factorial Gorenstein Fano fivefold. Moreover, it can be checked by hand that its singular locus is one-dimensional, consisting only of terminal points, and given by 
$$\Sing(X)=V(e_1,e_2,e_3,u_E)\cup V(e_1,e_2,e_3,e_4,u_6).$$

The wall relation $e_5+(-1)u_E+u_6=0$ determines an extremal contraction $\pi:X\to Y$ sending the Weil divisor $V(u_E)\subseteq X$ (which is not Cartier) onto $A=V(e_5,u_6)\subseteq Y$. Finally, the relation
$$e_1+e_2+e_3+2e_4+3e_5+u_6=0 $$
implies that $Y\cong \mathbb{P}(1^4,2,3)$, which is not Gorenstein, by \cite[Lemma 3.5.6]{CK99}. 
\end{exem}

\section{Toric universal coverings in codimension 1}\label{section:toric1coverings}

Let us recall that, every complete $\mathbb{Q}$-factorial toric variety of Picard number one can be seen as the quotient of a Weighted Projective Space by the action of a finite group such that the corresponding quotient map is {\it quasi-\'etale}, i.e., a finite surjective morphism which is \'etale in codimension 1 (see Definition \ref{defi:quasietale} and Example \ref{FWPS}). In this sense, Weighted Projective Spaces are the simplest singular $\mathbb{Q}$-factorial toric varieties of Picard number one. In the context of classification of toric Fano varieties satisfying the condition $(\dag)$, we would like to consider similar covers, but without the restriction on the Picard number. This is because quasi-\'etale morphisms have no ramification divisor and thus we can easily show that the resulting covering variety will be Fano and will satisfy the property $(\dag)$ (see Proposition \ref{coveringsing}). Therefore, in principle, it would be enough to consider the simplest toric varieties from the point of view of these quasi-\'etale morphisms in order to give a classification for the general case. We also refer the reader to the recent article of D. Greb, S. Kebekus and T. Peternell \cite{GKP16} where they constructed and studied these spaces for klt pairs.

The aim of this section is therefore to recall some recent results due to M. Rossi and L. Terracini concerning the construction and combinatorics of these quasi-\'etale universal covers in the toric setting and to describe the lifting of the extremal contractions appearing in the previous sections to these spaces. 

Let us recall some of the results and definitions introduced in \cite{Buc02}. We will follow the terminology of {\it quasi-\'etale morphisms}, introduced by F. Catanese \cite{Cat07}.

\begin{defi}[Quasi-\'etale morphism]\label{defi:quasietale}
Let $X$ be a complex normal algebraic variety. A {\it quasi-\'etale} morphism (or a {\it 1-covering}) is a finite surjective morphism $\varphi:Y\to X$ which is unramified in codimension 1. Namely, there exists a subvariety $V\subseteq X$ of codimension $\operatorname{codim}_X(V) \geq 2$ such that $\varphi|_{\varphi^{-1}(X\setminus V)}:\varphi^{-1}(X\setminus V)\to X\setminus V$ is \'etale.

Moreover, a {\it universal quasi-\'etale morphism} is a quasi-\'etale morphism $\varphi:Y \to X$ which is universal in the sense that for any quasi-\'etale morphism $f:Z\to X$ there exists a (not necessarily unique) quasi-\'etale morphism  $g:Y\to Z$ such that $\varphi=f\circ g$.
\end{defi}

\begin{prop}[{\cite[Corollary 3.10, Remark 3.14]{Buc02}}] A quasi-\'etale morphism $\varphi:Y\to X$ is universal if and only if $\pi_1(Y_{\lisse})$ is trivial.
\end{prop}

\begin{prop}\label{coveringsing} Let $\varphi:Y\to X$ be a quasi-\'etale morphism between normal projective varieties. Then, 
\begin{enumerate}
\item If $K_X$ is a Cartier divisor, then $K_Y$ is a Cartier divisor.
\item If $X$ is a Fano variety, then $Y$ is a Fano variety.
\item If $X$ has terminal (resp. canonical) singularities, then $Y$ also has terminal (resp. canonical) singularities.
\end{enumerate}
\end{prop}

\begin{proof}
As $\varphi:Y\to X$ is unramified in codimension 1, there is no ramification divisor and hence $\varphi^*K_X = K_Y$, implying (1). The point (2) follows from \cite[Proposition 5.1.12]{EGAII}, while (3) follows from \cite[Proposition 3.16]{Kol97}.
\end{proof}

Moreover, we have the following characterization of the fundamental group of the smooth locus for $\mathbb{Q}$-factorial toric varieties, which follows from \cite[Corollary 3.10, Theorem 4.8]{Buc02} and \cite[Theorem 2.4]{RT15a}.

\begin{theo}\label{toric pi_1^1} 
Let $X$ be a $\mathbb{Q}$-factorial toric variety defined by the fan $\Delta_X \subseteq N_\mathbb{R}$ and let $N_{\Delta_X(1)} \subseteq N$ be the sub-lattice of $N$ generated by the primitive lattice generators $u_\rho \in N$ of all the rays $\rho \in \Delta_X(1)$. Then,
$$\pi_1(X_{\lisse})\cong N/N_{\Delta_X(1)}\cong \operatorname{Tors}(\operatorname{Cl}(X)). $$

\end{theo}

\begin{exem}\label{FWPS} Let $X=X(\Delta_X)$ be a complete $\mathbb{Q}$-factorial toric variety of dimension $n$ such that $\rho_X=1$. Then, we will say that $X$ is a Fake Weighted Projective Space. 

This name comes from the following observation: the fan $\Delta_X$ has cone generators $u_0,\ldots,u_n\in \Delta_X(1)$, and the maximal cones of $\Delta_X$ are generated by the $n$-element subsets of $\{u_0,\ldots,u_n\}\subseteq N\cong \mathbb{Z}^n$. As they are linearly dependent,
$$\sum_{i=0}^n \lambda_i u_i = 0, $$
for some $\lambda_0, \ldots, \lambda_n\in \mathbb{Z}_{\geq 1}$. Therefore, $\pi_1(X_{\lisse})=\{0\}$ if and only if the primitive lattice vectors $u_0,\ldots,u_n\in N$ generate the lattice $N$. If it is the case we will have that $X\cong \mathbb{P}(\lambda_0,\ldots,\lambda_n)$, by \cite[Example 5.1.14]{CLS}.
\end{exem}

Following \cite{RT15a} and \cite{RT15b}, it is natural to consider $\mathbb{Q}$-factorial complete toric varieties with torsion-free class group as analogs of Weighted Projective Spaces.

\begin{defi}\label{PWS}
Let $X=X(\Delta_X)$ be a $\mathbb{Q}$-factorial complete toric variety of dimension $n$. We define the {\it canonical quasi-\'etale universal cover of $X$} to be the quasi-\'etale morphism $\pi_X:\widehat{X}\to X$ corresponding to the map of fans 
$$ \Pi_X:(N_{\Delta_X(1)},\Delta_X)\to (N_X,\Delta_X). $$
Moreover, we say that $X$ is a Poly Weighted Space (PWS) if 
$$\pi_1(X_{\lisse})\cong N/N_{\Delta_X(1)}\cong \operatorname{Tors}(\operatorname{Cl}(X))\cong \{0\}.$$
\end{defi}

After the recent works of M. Rossi and L. Terracini, there is an explicit combinatorial construction (via Gale duality) of the canonical quasi-\'etale universal cover of any $\mathbb{Q}$-factorial complete toric variety. This extends Example \ref{FWPS} and \cite[Theorem 6.4]{Buc02} to higher class group rank varieties (see \cite[Theorem 2.2]{RT15b} for details).

\vspace{2mm}

The remaining of the section will be devoted to study contractions $X\to Y$ of extremal rays as in the previous section via universal quasi-\'etale morphisms, and without the assumption of isolated singularities in the divisorial case. 

\vspace{2mm}

The case of extremal contractions of fiber type was studied by Y. Kawamata in \cite[Lemma 4.1]{Kaw06} and it is commonly known as {\it Kawamata's covering trick}. In our context, Proposition \ref{thm:rho=2 toric fibration} and Kawamata's covering trick specialize to the following result.

\begin{coro}\label{prop:1-covering,fiber type}
 Let $X$ be a toric Fano variety satisfying $(\dag)$. Assume that there exists an effective prime divisor $D\subseteq X$ such that $\dim_\mathbb{R}\N_1(D,X)=1$ and that $\rho_X=2$. Let $R\subseteq \overline{\NE}(X)$ be an extremal ray such that $D\cdot R > 0$ and let us denote by $\pi:X\to Y$ the corresponding extremal contraction. Assume that $\pi$ is of fiber type. Then there exist weights $\lambda_0,\ldots,\lambda_{n-1}\in \mathbb{Z}_{>0}$ and a cartesian diagram of toric varieties
$$
\xymatrix{
\widehat{X} \ar[r]^{\widehat{\pi}\hspace{1cm}} \ar[d]_{\pi_X} & \mathbb{P}(\lambda_0,\ldots,\lambda_{n-1}) \ar[d]^{\pi_Y} \\
X \ar[r]^\pi & Y
}
$$
where vertical arrows denote the corresponding canonical quasi-\'etale universal covers, and $\widehat{X}$ is a Gorenstein Fano PWS with terminal singularities such that $\rho_{\widehat{X}}=2$. Moreover, $\widehat{\pi}:\widehat{X}\to \mathbb{P}(\lambda_0,\ldots,\lambda_{n-1})$ leads to an isomorphism 
$$\widehat{X}\cong \mathbb{P}(\mathcal{O}_{\mathbb{P}(\lambda_0,\ldots,\lambda_{n-1})}\oplus \mathcal{O}_{\mathbb{P}(\lambda_0,\ldots,\lambda_{n-1})}(a)). $$
\end{coro}

\begin{proof}
 By Proposition \ref{thm:rho=2 toric fibration}, $\pi:X\to Y$ is a $\mathbb{P}^1-$bundle and $X$ is isomorphic to $\mathbb{P}_Y(\mathcal{O}_Y\oplus \mathcal{O}_Y(a))$, where $Y$ is a $\mathbb{Q}$-factorial Gorenstein Fano variety with terminal singularities. Moreover, \cite[Theorem 2.4, Corollary 2.5]{Re83} implies that we can always suppose that $\pi$ is induced by the projection $\Pi: \mathbb{Z}^n \to \mathbb{Z}^{n-1}\times \{0\},\;(x_1,\ldots,x_n)\mapsto (x_1,\ldots,x_{n-1},0)$.
 
 Since $\pi$ is locally trivial with reduced fibers, it follows from \cite[Remark 3.3, Remark 3.8]{CD08} that if $\sigma=\operatorname{cone}(u_1,\ldots,u_{n-1},e_n)$ and $\sigma'=\operatorname{cone}(u_1,\ldots,u_{n-1},-e_n)$ are maximal cones in $\Delta_X(n)$ which are sent by $\Pi$ onto a maximal cone ${\tau=\operatorname{cone}(t_1,\ldots,u_{t-1}) \in \Delta_Y(n-1)}$ then we have that, up reordering if necessary, $\Pi(u_i)=t_i$ for $i=1,\ldots,n-1$. 
 
 This holds for every invariant open affine open subset of $Y$ and hence it follows that there is an induced morphism between the canonical quasi-\'etale universal covers, $\widehat{\pi}:\widehat{X}\to \widehat{Y}$ (cf. \cite[Lemma 4.1]{Kaw06}), which is an extremal contraction of fiber type. Moreover, the induced commutative diagram is cartesian in the category of schemes by \cite[Lemma 2.2.7]{Mol16}.
 
 Finally, we have that $\widehat{Y}\cong \mathbb{P}(\lambda_0,\ldots,\lambda_{n-1})$ for some weights $\lambda_0,\ldots,\lambda_{n-1}\in \mathbb{Z}_{>0}$ since $\rho_Y=1$, while Proposition \ref{coveringsing} and Proposition \ref{thm:rho=2 toric fibration} imply that $\widehat{X}$ and $\widehat{Y}$ are Fano Gorenstein varieties with terminal singularities and that there is an isomorphism $\widehat{X}\cong \mathbb{P}(\mathcal{O}_{\mathbb{P}(\lambda_0,\ldots,\lambda_{n-1})}\oplus \mathcal{O}_{\mathbb{P}(\lambda_0,\ldots,\lambda_{n-1})}(a))$.
\end{proof}

\begin{exem}
Let $X$ as in Corollary \ref{prop:1-covering,fiber type} and suppose that $\operatorname{Tors}(\operatorname{Cl}(X))\cong\{0\}$, i.e., that $X\cong \widehat{X}$. The extremal contraction of fiber type
$$ X\to \mathbb{P}(\lambda_0,\ldots,\lambda_{n-1})$$
leads to an isomorphism $X\cong \mathbb{P}(\mathcal{O}_{\mathbb{P}(\lambda_0,\ldots,\lambda_{n-1})}\oplus \mathcal{O}_{\mathbb{P}(\lambda_0,\ldots,\lambda_{n-1})}(a))$. Then,
\begin{itemize}
\item[(a)] $X$ is Gorenstein $\Leftrightarrow$ $\mathbb{P}(\lambda_0,\ldots,\lambda_{n-1})$ is Gorenstein $\Leftrightarrow$ $\lambda_i|h$ for every $i\in\{0,\ldots,n-1\}$, by \cite[Lemma 3.5.6]{CK99}.
\item[(b)] $X$ is terminal $\Leftrightarrow$ $\mathbb{P}(\lambda_0,\ldots,\lambda_{n-1})$ is terminal $\Leftrightarrow$ $\sum_{i=0}^n\left\{\lambda_i\kappa\slash h \right\}\in \{2,\ldots,n-1\}$ for each $\kappa\in \{2,\ldots,h-2\}$, by \cite[Proposition 2.3]{Kas13}.
\item[(c)] $X$ is Fano $\Leftrightarrow$ $0 \leq a \leq i_{\mathbb{P}(\lambda_0,\ldots,\lambda_{n-1})}-1= h -1, $ by Proposition \ref{thm:rho=2 toric fibration} and the formula \cite[Proposition 2.3]{Mori75} for the canonical divisor of $\mathbb{P}(\lambda_0,\ldots,\lambda_{n-1})$.
\end{itemize}
Here, $h=\sum_{i=0}^{n-1}\lambda_i$ and $\{x\}$ denotes the fractional part of $x\in \mathbb{R}$.
\end{exem}

The case of divisorial extremal contractions follows in a similar way.

\begin{lemm}\label{divisorialcovering}
Let $X$ be a $\mathbb{Q}$-factorial projective toric variety and $R\subseteq \overline{\NE}(X)$ an extremal ray defining a divisorial contraction $\varphi_R:X\to X_R$. Then, there is a commutative diagram of toric morphisms
$$
\xymatrix{
\widehat{X} \ar[r]^{\widehat{\varphi_R}} \ar[dd]_{\pi_X} & \widetilde{X_R} \ar[rd]^{\mu} & \\
 & & \widehat{X_R} \ar[ld]^{\pi_{X_R}}\\
X \ar[r]^{\varphi_R} & X_R &
}
$$
that satisfies the following conditions:
\begin{itemize}
\item[(a)] $\pi_X$ and $\pi_{X_R}$ are the corresponding canonical quasi-\'etale universal covers.
\item[(b)] $\mu$ is a quasi-\'etale morphism given by the inclusion of lattices ${N_{\Delta_{X_R}(1)}\subseteq N_{\Delta_X(1)} }$ of index $d_E\geq 1$, where $d_E$ is defined by the condition that the integral generator of ${\mathbb{R}_{\geq 0} u_E \cap N_{\Delta_{X_R}(1)} }$ is $d_Eu_E$, where $V(u_E)$ is the exceptional divisor of $\varphi_R$.
\end{itemize}
Moreover,
\begin{itemize}
\item[(c)] $\widehat{\varphi_R}:\widehat{X}\to \widetilde{X_R}$ is a divisorial contraction with $(\pi_X)_* \Exc(\widehat{\varphi_R})=\Exc(\varphi_R)$ and $(\pi_{X_R}\circ \mu)_* \widehat{\varphi_R}(\Exc(\widehat{\varphi_R}))=\varphi_R(\Exc(\varphi_R))$.
\item[(d)] If $X$ (resp. $X_R$) is a Fano variety then $\widehat{X}$ (resp. $\widehat{X_R}$ and $\widetilde{X_R}$) is.
\item[(e)] If $X$ (resp. $X_R$) has Gorenstein singularities then $\widehat{X}$ (resp. $\widehat{X_R}$ and $\widetilde{X_R}$) does.
\item[(f)] If $X$ has terminal (resp. canonical) singularities, then all varieties in the diagram have terminal (resp. canonical) singularities.
\end{itemize}
\end{lemm}

\begin{proof}
Let us suppose that $\varphi_R:X\to X_R$ is given by the contraction of the wall $\omega=\operatorname{cone}(u_1,\ldots,u_{n-1})$ separating the maximal cones $\sigma=\operatorname{cone}(u_1,\ldots,u_{n-1},u_n)$ and $\sigma'=\operatorname{cone}(u_1,\ldots,u_{n-1},u_{n+1})$. Then, the wall relation satisfied by these cones (defining the contraction) is given by
$$b_u u_n + \sum_{i=1}^{n-1} b_iu_i+b_{n+1}u_{n+1}=0, $$
where $b_n,b_{n+1}\in \mathbb{Z}_{>0}$ and $b_i\in \mathbb{Z}$. 

Since $\varphi_R$ is a divisorial contraction we can suppose that (up to reordering, if necessary) $b_1<0$ and $b_2,\ldots,b_{n-1}\geq 0$, by \cite[Theorem 2.4, Corollary 2.5]{Re83}. Thus, $E=\Exc(\varphi_R)=V(u_1)$ and the contraction corresponds to the stellar subdivision of the cone
$$\sigma=\operatorname{cone}(u_2,\ldots,u_n,u_{n+1})\in \Delta_{X_R}(n) $$
with respect to the primitive lattice vector $u_1$ satisfying the wall relation above.

The canonical quasi-\'etale universal cover of $X$ (resp. $X_R$) is given by the fan $\Delta_X$ (resp. $\Delta_{X_R}$) but seen in the sub-lattice $N_{\Delta_X(1)}$ (resp. $N_{\Delta_{X_R}(1)}$) of $N$.

Clearly we have the inclusion of lattices $N_{\Delta_{X_R}(1)}\subseteq N_{\Delta_X(1)}$, which is of finite index since $(-b_1)u_1 \in N_{\Delta_{X_R}(1)}$, by the wall relation above. Hence, we obtain an induced quasi-\'etale morphism $\mu:\widetilde{X_R}\to\widehat{X_R}$ by \cite[Lemma 3.3]{AP13}.

Now, the fan of $\widehat{X}$ is obtained by the stellar subdivision of the fan of $\widetilde{X_R}$ with respect to the primitive vector $u_1$ satisfying the wall relation above, obtaining the desired commutative diagram that satisfies (a), (b) and (c) by construction.

The last three assertions follows from Proposition \ref{coveringsing} and \cite[Corollary 3.43]{KM98}.
\end{proof}

We can now prove Proposition \ref{prop:1-covering,divisorial}.

\begin{proof}[Proof of Proposition \ref{prop:1-covering,divisorial}]
The situation in Lemma \ref{divisorialcovering} above becomes simpler in this case because the extremal contraction $\pi:X\to Y$ is induced by a wall relation of the form $u_n+u_{n+1}=u_1$, by Lemma \ref{toricdivisorial}. Therefore, we have that $N_{\Delta_Y(1)}= N_{\Delta_X(1)}$ and hence $\widetilde{Y}\cong \widehat{Y}$ with the notation as in Lemma \ref{divisorialcovering}. In other words, there is an induced morphism between the canonical quasi-\'etale universal covers, $\widehat{\pi}:\widehat{X}\to \widehat{Y}$, which is a divisorial extremal contraction. Moreover, the induced commutative diagram is cartesian in the category of schemes by \cite[Lemma 2.2.7]{Mol16}.
 
Finally, we have that $\widehat{Y}\cong \mathbb{P}(\lambda_0,\ldots,\lambda_n)$ for some weights $\lambda_0,\ldots,\lambda_n\in \mathbb{Z}_{>0}$, by Example \ref{FWPS}. The result follows now directly from the second part of Lemma \ref{divisorialcovering}.
\end{proof}

\begin{exem}

Let $X$ as in Proposition \ref{prop:1-covering,divisorial} and suppose that $\operatorname{Tors}(\operatorname{Cl}(X))\cong \{0\}$, i.e., that $X\cong \widehat{X}$. The extremal divisorial contraction
$$\pi:X\to \mathbb{P}(\lambda_0,\ldots,\lambda_n) $$
determines the shape of the fan of $X$ in terms of the fan of $\mathbb{P}(\lambda_0,\ldots,\lambda_n)$ (which is well known):

The fan of $\mathbb{P}(\lambda_0,\ldots,\lambda_n)$ is given by $n+1$ lattice primitive vectors $u_0,\ldots,u_n$ that generates the lattice $N$ and that satisfy the relation
$$\sum_{i=0}^n \lambda_i u_i=0. $$
Let us suppose that $\pi$ contracts $E=V(u_E)\subseteq X$ onto the invariant subvariety $A=V(u_i,u_j)\subseteq \mathbb{P}(\lambda_0,\ldots,\lambda_n)$, of codimension two. Then, $u_E=u_i+u_j$, by Lemma \ref{toricdivisorial}.

In particular, the same computation used to prove \cite[Lemma 3.5.6]{CK99} shows that $X$ is a Fano Gorenstein variety if and only if $\lambda_i|h$, $\lambda_j|h$, $\lambda_k|(h-\lambda_i)$ and $\lambda_k|(h-\lambda_j)$ for every $k\neq i,j$, where $h=\sum_{i=0}^n\lambda_i$.

However, to the best of the author's knowledge, the characterization such $X$ having terminal singularities is more subtle. Indeed, if $X$ is a $\mathbb{Q}$-factorial Fano Gorenstein toric variety then it corresponds to a simplicial reflexive lattice polytope $P\subseteq N_\mathbb{R}$ (see \cite{Bat94} for details). In \cite[Corollary 3.7]{Nill05}, B. Nill characterizes all polytopes among these ones that correspond to varieties with only terminal singularities, but it does not seem easy to translate this characterization into a function of the weights $\lambda_0,\ldots,\lambda_n$.

Finally, it should be noticed that if one of the weights is equal to 1, say $\lambda_0=1$, then we have a coordinate-wise description of the primitive vectors defining the fan of $\mathbb{P}(1,\lambda_1,\ldots,\lambda_n)$. Namely, the canonical basis of $\mathbb{Z}^n$ together with the vector $(-\lambda_1,\ldots,-\lambda_n)\in \mathbb{Z}^n$. In this case, we can explicitly compute the Cartier data $\{m_\sigma\}_{\sigma\in \Delta_X(n)}\subseteq M$ of $K_X$, allowing us to decide whether the singularities of $X$ are terminal or not.

In particular, we compute that $X$ has only Gorenstein terminal singularities if all the integers
$$\dfrac{h}{\lambda_i},\dfrac{h}{\lambda_j},\dfrac{h-\lambda_i}{\lambda_k},\dfrac{h-\lambda_j}{\lambda_k} $$
considered before, are equal or greater than 3. The variety defined in Example \ref{contreexample} satisfy this condition.
\end{exem}

\end{document}